\documentclass{amsart}
\usepackage{amsmath,amssymb,comment}
\usepackage{amsfonts}
  \usepackage{paralist}
\usepackage{latexsym}
\usepackage{epsfig}
\usepackage{graphicx}
\usepackage{MnSymbol}
\usepackage{url}
 \usepackage[colorlinks=true]{hyperref}

\hypersetup{urlcolor=blue, citecolor=red}

\newtheorem{theorem}{Theorem}[section]

\newtheorem{claim}[theorem]{Claim}

\newtheorem{conjecture}[theorem]{Conjecture}
\newtheorem{corollary}[theorem]{Corollary}

\newtheorem{lemma}[theorem]{Lemma}

\newtheorem{proposition}[theorem]{Proposition}
\newtheorem{remark}[theorem]{Remark}

\newcommand{\comb}[2]{\mbox{ $\left(\!\!\begin{array}{c} {#1} \\
{#2}\end{array}\!\!\right)$}}
\newcommand{\norm}[1]{\Vert {#1} \Vert}

\def\comp{{\circ}}

\def\N{\mathbb{N}}
\def\R{\mathbb{R}}
\def\C{\mathbb{C}}
\def\esc{\lambda}

\def\esc{\lambda}
\def\e{\varepsilon}

\def\RR{\tilde{R}}
\def\aa{\tilde{a}}
\def\bb{\tilde{b}}
\def\KK{\tilde{K}}
\def\GG{\tilde{G}}
\def\BB{\tilde{B}}

\def\Dif{{D}}
\def\Spec{{\rm Spec}\,}
\def\Re{{\rm Re}\,}

\def\L{\mathcal{L}}
\def\O{\mathcal{O}}
\def\U{\mathcal{U}}
\def\A{\mathcal{A}}
\def\CC{\kappa}

\def\o{o}
\newcommand{\Id}{\mathrm{Id}}
\newcommand{\CCC}{\mathcal{C}}

\title[Gevrey parabolic manifolds]{
Gevrey estimates for one dimensional parabolic invariant manifolds
of non-hyperbolic fixed points}

\author[Inmaculada~Baldom\'a, Ernest~Fontich and Pau~Mart\'{\i}n]{}
\email{immaculada.baldoma@upc.edu}
\email{fontich@ub.edu}
\email{p.martin@upc.edu}

\thanks{
I.B and P.M. have been partially supported by the Spanish
MINECO-FEDER Grant MTM2015-65715-P and the Catalan Grant 2014SGR504.
The work of E.F. has been partially supported by the Spanish
Government grant MTM2013-41168P and the Catalan Government grant
2014SGR-1145.}

\begin{document}

\begin{abstract}
    We study the Gevrey character of a natural parameterization of one dimensional
     invariant manifolds associated to a \emph{parabolic direction} of fixed points of analytic maps,
     that is, a direction associated with an eigenvalue equal to~$1$.
     We show that, under general hypotheses, these invariant manifolds are Gevrey
     with type related to some explicit constants. We provide examples of the optimality
     of our results as well as some applications to celestial mechanics, namely, the Sitnikov
     problem and the restricted planar three body problem.
\end{abstract}

\maketitle

\centerline{\scshape Inmaculada Baldom\'a}
\medskip
{\footnotesize
 \centerline{Departament de Matem\`{a}tiques,}
\centerline{Universitat Polit\`{e}cnica de Catalunya,}
\centerline{Av. Diagonal 647, 08028 Barcelona, Spain}
}
\medskip

\centerline{\scshape Ernest Fontich}
{\footnotesize
 \centerline{Departament de Matem\`{a}tiques i Inform\`atica,}
\centerline{Universitat de Barcelona,}
\centerline{Gran Via 585,
08007, Barcelona, Spain} }
\medskip

\centerline{\scshape Pau Mart\'{\i}n}
{\footnotesize
 \centerline{Departament de Matem\`{a}tiques,}
\centerline{Universitat Polit\`{e}cnica de Catalunya}
\centerline{Ed.~C3, Jordi Girona 1--3, 08034 Barcelona, Spain}
}

\bigskip

 \centerline{(Communicated by the associate editor name)}

\section{Introduction}

Let us consider a dynamical system defined through a map $F: \R^{m}
\to \R^{m}$ with a fixed point at the origin. To each invariant
subspace $E$ of $DF(0)$ one can try to identify its corresponding
counterpart for $F$, that is, a manifold tangent to $E$ at the
origin invariant by~$F$, if it exists. Of course, these invariant
manifolds need  not be unique, or even if they do exist, they can be
less regular than the map~$F$, depending on the resonance relations
in $\Spec DF(0)_{\mid E}$. In the case that $F$ is analytic or
$C^{\infty}$, one can even ask if there exists a \emph{formal}
invariant manifold tangent to~$E$, that is, a formal  power series
which solves at all orders an appropriate invariance equation.

One way to obtain manifolds invariant by~$F$ is by using the
parameterization method. A brief description is the following. If
$E\subset \R^{m}$ is a subspace of dimension~$n$, invariant by
$DF(0)$, one can try to find an invariant manifold by~$F$ tangent
to~$E$ at the origin as an embedding $K:B_{\rho} \subset \R^n \to
\R^m$ (here $B_{\rho}$ denotes the ball of radius~$\rho$) such that
$K(0) = 0$, $DK(0) \R^n = E$ and  a \emph{reparameterization}
$R:B_{\rho} \to \R^n$, $R(0) =0$, satisfying the \emph{invariance
equation}
\begin{equation}
\label{eq:inv_equation} F \circ K = K \circ R.
\end{equation}

Well known examples of invariant manifolds are the strong stable and
unstable manifolds, which, roughly speaking, are associated to the
eigenvalues $\lambda$ of $DF(0)$ such that $|\lambda| > \mu
>1$ and $|\lambda| < \nu <1 $, respectively, for given constants $\mu$ and $\nu$.
See, for instance, \cite{HirschP70,Irwin70,Irwin80}
and~\cite{CabreFL03a,HaroCFLM16,CabreFL05} and the references
therein. These manifolds are as regular as the map in a neighborhood
of the fixed point. In particular, analytic if so is the map. Their
expansions in power series are convergent.

When one considers invariant manifolds tangent to subspaces
associated to subsets of non-resonant eigenvalues the situation
becomes more interesting. The invariance equation can be solved at
all orders, due to the non-resonant character of the eigenvalues.
This solution provides a formal invariant manifold. In general, this
formal series corresponds to a regular meaningful object if one
imposes the non-resonant eigenvalues to be of modulus larger (resp.
smaller) than one. That is, when the non-resonant manifolds are
submanifolds of the strong unstable (resp. stable) manifold. See for
instance \cite{CabreFL03a}. If the map is analytic, these
non-resonant manifolds are also analytic and, again, their
expansions are convergent.

Here we consider the totally resonant case, that is, manifolds
tangent to subspaces associated to the eigenvalue $1$ and, thus,
submanifolds of the center manifold. We call these manifolds
\emph{parabolic}.

When the map is tangent to the identity at the fixed point, that is,
$DF(0) = \Id$, any subspace of $\R^{m}$ is invariant by $DF(0)$. In
order to identify the subspaces which are susceptible to have an
invariant manifold tangent to them it is necessary to pay attention
to the first next non-vanishing terms of the Taylor expansion of $F$
at the origin. This is the case considered in~\cite{BFdLM2007}, when
one looks for one dimensional manifolds. See also~\cite{McGehee73}.
In the latter, only analytic manifolds where considered, while the
former includes the case of finite differentiability. The former
also includes the construction of formal solutions of the invariance
equation~\eqref{eq:inv_equation}. Under the conditions
in~\cite{BFdLM2007}, if the map $F$ is analytic or $\CCC^{\infty}$,
the parabolic invariant manifolds exist and are $\CCC^{\infty}$ at
the fixed point. See also~\cite{Hakim98} and the
survey~\cite{Abate15} in the setting of complex dynamics.

In~\cite{BF2004} it is studied the case of $F$ analytic, tangent to
the identity and with invariant manifolds of dimension two or
greater. These manifolds are analytic in their domain, although in
general the fixed point is only at their boundary. In this case,
however, it is easy to see that in general there are no formal
solutions (in the sense of power series) of the invariance
equation~\eqref{eq:inv_equation}. In the same setting,
in~\cite{BFM2015b} it is shown that the invariant manifolds can be
approximated by sums of homogeneous functions of increasing order.

In the present paper we assume that $F$ is an analytic local
diffeomorphism in a neighborhood of the origin in $\R\times \R^d
\times \R^{d'}$ and satisfies
\begin{equation}\label{linearpartF}
DF(0) = \begin{pmatrix} 1& 0 & 0 \\
0 & \Id_{d} & 0 \\
0& 0 & C
\end{pmatrix}
\end{equation}
with $1 \not \in \Spec C$ and $\Id_{d}$ is the identity matrix in
$\R^{d}$. When $d=0$, that is, when $1$ is a simple eigenvalue of
$DF(0)$, this class of maps was studied in~\cite{BH08}. There the
authors proved that if the map $F$ has the form
\[
F:\begin{pmatrix} x \\ z \end{pmatrix} \in \R \times \R^{d'} \mapsto
\begin{pmatrix} x - a x^N + z \O_{N-1} + \O_{N+1} \\ Cz + \O_2
\end{pmatrix},
\]
where $\O_j$ stands for $\O(\|(x,z)\|^j)$,  with $a \neq 0$, $N\ge
2$,
 the invariance equation admits a formal solution $\hat K(t) = \sum_{k\ge 1} K_j t^j$,
$K_j \in \R^{1+d'}$, with some polynomial  reparameterization~$R$,
and that the series is $\alpha$-Gevrey with $\alpha = 1/(N-1)$, that
is, there exist constants $c_1,c_2>0$ such that
\[
\norm{K_j} \leq c_1 c_2^j j!^{\alpha},  \qquad j \ge 0.
\]
Furthermore, if $a >0$ and $\Spec C \subset \{ z\in \C\mid |z| \ge
1, \, z\neq 1\}$,  there is an analytic solution $K$ of the
invariance equation, defined in some convex set $V$ with
$0\in\partial V$, that is $\alpha$-Gevrey asymptotic to $\hat K$,
that is, there exist constants $c_1,c_2>0$ such that
$$
\left \| K(t) - \sum_{j=1}^{n-1} K_j t^j \right \| \leq c_1 c_2^n
n!^{\alpha} |t|^n, \qquad  n\ge 1, \quad t\in V.
$$

Here we generalize these results to the case $d>0$, $d'\ge 0$. That
is, if the map $F$ has the linear part~\eqref{linearpartF} and
certain conditions on the nonlinear terms are met (see
Theorem~\ref{formal_theorem}), the invariance
equation~\eqref{eq:inv_equation} for the map $F$ admits a formal
solution~$\hat K (t)$, which is $\gamma$-Gevrey for a precise
$\gamma$ (defined in~\eqref{defalpha}). We provide examples for
which this value of $\gamma$ is sharp, that is, $\hat K (t)$ is not
$\gamma'$-Gevrey for any $0\le\gamma'< \gamma$. These conditions can
be seen as non-resonances, because they allow to solve some
cohomological equations (see also
Claim~\ref{claim:autonomousflowsanalyticmanifolds}). Also they would
imply the existence of a \emph{characteristic direction}, if the map
was truly tangent to the identity at the fixed point.

Adding some additional conditions (see
Theorem~\ref{mtexistencesolution}), we also prove that there is a
true invariant manifold given by an analytic parameterization $K$
which is $\gamma$-Gevrey asymptotic the the formal series $\tilde K$
in some complex convex set with $0$ at its boundary. We will refer to this
manifold as a parabolic manifold and we notice that the information
about its internal dynamics is given by $R(t)$ which, in our case,
turns out to be a polynomial. Depending on $R$ the parabolic
manifold may behave as a (weak) stable manifold (in the sense that
the iterates of its points converge to de origin) or a (weak)
unstable manifold. In those cases we will denote them by parabolic
stable/unstable manifolds.

Of course, the conditions that allow the existence of a formal
solution are weaker than the ones we need to impose in order to have
a true invariant manifold. However, we prove that if the map
possesses a one dimensional parabolic stable invariant manifold to
the origin, tangent to a particular direction associated to an
eigenvalue equal to~$1$, and it is non-degenerate (in the sense of
Proposition~\ref{prop:FormaF}), then there are suitable coordinates
in which the map satisfies our conditions (listed in~\eqref{system}
and hypotheses below).

Our results provide upper bounds for the coefficients of the
asymptotic expansion of the invariant manifold. The existence of
lower bounds remains open. Although we provide examples that show
the optimality of our results, we also prove that if the map is the
time one map of an autonomous analytic vector field, satisfying our
hypotheses, the invariant manifold, when written as a graph, extends
analytically to a neighborhood of the origin (see
Claim~\ref{claim:autonomousflowsanalyticmanifolds}). That is, the
invariant manifolds can be more regular than what we claim. This is
no longer true for the stroboscopic Poincar\'e map of time-periodic
equations (Claim~\ref{example:optimal_Gevrey_N_equal_M}). However,
although obtaining lower bounds is out of the scope of the present
work, we show in Proposition~\ref{prop:trickylowerbounds} that the
conditions to obtain such lower bounds cannot depend on a finite
number of coefficients of the Taylor expansion of the map~$F$.

An important consequence of our present results is the Gevrey
character of some invariant manifolds in some problems of Celestial
Mechanics. In several instances of the restricted three body, like
the Sitnikov problem or the restricted planar three body problem,
the parabolic infinity is foliated by periodic orbits. The
associated stroboscopic Poincar\'e map satisfies the conditions of
our existence result (Theorem~\ref{mtexistencesolution}) with
$d'=0$, which implies that the manifolds are at least $1/3$-Gevrey
at the origin. See Section~\ref{sec:celestial_mechanics} for more
details. Sim\'o and Mart{\'{\i}}nez announced in 2009
\cite{MartinezS09} that, in the case of the Sitnikov problem, the
manifolds are precisely $1/3$-Gevrey, which would imply the
optimality of our result in the sense that these manifolds are not
more regular. The numerical experiments in~\cite{MartinezS14}
strongly support the same claim for the restricted circular planar
three body problem. These computations and the example we provide in
Claim~\ref{example:optimal_Gevrey_N_equal_M} move us to conjecture
that the invariant manifolds of infinity of the restricted three
body problem are exactly $1/3$-Gevrey (see
Conjecture~\ref{conjecture} for the precise statement).

The structure of the paper  is as follows. In
Section~\ref{sec:notation} we introduce the definitions and
notations we will use along the paper. In
Section~\ref{sec:Main_Results} we collect the main results of the
paper. Section~\ref{sec:examples} is devoted to present some
examples that show that the Gevrey order we find is optimal. We also
show how our theorems apply to the restricted three body problem.
The rest of the paper contains the proofs of the results on
Section~\ref{sec:Main_Results}. In Section~\ref{sec:formal} we
obtain the formal solution of the invariance equation. Its Gevrey
character is studied in Section~\ref{sec:Gevrey}. The existence of
the true manifold is proved in Section~\ref{sec:existence}. The
appendix contains the proofs of Propositions~\ref{prop:FormaF}
and~\ref{prop:trickylowerbounds}.

\section{Set up and notation}
\label{sec:notation} Let $\U\subset \R\times \R^{d} \times \R^{d'}$ be an
 open neighborhood of $0= (0,0,0)$. We consider $F:  {\U}\longrightarrow \R\times
\R^{d} \times \R^{d'}$, the real analytic maps
%
%
defined by
\begin{equation}\label{system}
F \begin{pmatrix} x \\ y \\ z
\end{pmatrix}
=
\begin{pmatrix} x - a x^N  +   f_N(x,y,z) + f_{\geq N+1} (x,y,z) \\
y  + x^{M-1} B_1 y+ x^{M-1} B_2 z+ g_{M}(x,y,z)+ g_{\geq
M+1}(x,y,z) \\
C z  +  h_{\geq 2}(x,y,z)
\end{pmatrix}
\end{equation}
where:
\begin{itemize}
\item $N,M\geq 2$ are integer numbers;
\item the constant $a$ is non-zero;
\item $1\not \in \Spec C$;
\item $f_N(x,y,z)$ is a homogeneous polynomial of degree $N$
such that $ f_N(x,0,0)= 0$. We introduce the notation
\[
v= \frac{1}{(N-1)!}
\partial^{N-1}_x
\partial_y f_N(0,0,0)\in \R^d, \quad   w= \frac{1}{(N-1)!}
\partial^{N-1}_x
\partial_z  f_N(0,0,0)\in \R^{d'},
\]
so that $\partial_y  f_N(x,0,0) = x^{N-1} v^\top$ and
$\partial_z  f_N(x,0,0) = x^{N-1} w^\top$;

\item $ g_{M}(x,y,z)$ is a homogeneous polynomial of degree $M$ such that
$ g_{M}(x,0,0)=0$, $\Dif_y  g_{M}(x,0,0)=0$ and $\Dif_z
g_{M}(x,0,0)=0$;

\item $f_{\geq N+1}$ has order $N+1$ (the function and its derivatives vanish
up to order $N$ at $(0,0,0)$), $g_{\geq M+1}$ has order~$M+1$ and
$h_{\geq 2}$ has order~$2$.
\end{itemize}
Since $F$ is real analytic, it can be extended to a complex
neighborhood ${\mathcal U}_\C$ of ${\mathcal U}$. For simplicity, we
will denote also  by $F$ this complex extension.
%

We introduce the following  notational conventions  we use
throughout the paper. We denote by $\hat{W}(t)= \sum_{k\geq 0} W_k
t^k$ any formal series in~$t$ and if $W(t)$ is a map, we denote by
$W_k=\frac{1}{k!} D^k W(0)$, if the derivatives are defined. The expressions
$W_{\leq l}$, $W_{\geq l+1}$, etc. will mean $\sum_{k=0}^{l} W_k
t^k$, $\sum_{k\geq l+1} W_k t^k$, etc., and we will use them without
further mention. The projection over the $x$, $y$ or $z$-component
is denoted by $\pi^x$, $\pi^y$ and $\pi^z$. If $W(\cdot) \in
\C^{1+d+d'}$ (or if $W$ is a map taking values in $\C^{1+d+d'}$, or
a power series with coefficients in $\C^{1+d+d'}$), we write $W^x=
\pi^x W$, $W^y= \pi^y W$ and $W^z= \pi^z W$. We also use $\pi^{x,y}W
= W^{x,y}=(W^x,W^y)$, or any other combination of the variables.

We finally introduce the constants
\begin{equation}\label{defalpha}
\alpha=\frac{1}{N-1} \qquad \text{and}  \qquad \gamma=
\begin{cases} \displaystyle\frac{1}{N-1}, &  \text{if $N\leq M$},\\
\displaystyle\frac{1}{N-M}, & \text{if $N>M$},
 \end{cases}
\end{equation}
which will play a capital role in our results.

\section{Main Results}
\label{sec:Main_Results} We start dealing with formal solutions of
the invariance equation  $F\circ K = K\circ R$.  We provide
conditions that ensure the existence of a formal solution as a power
series, which turns out to be $\gamma$-Gevrey.

\begin{theorem}\label{formal_theorem}
Let $F$ be a map of the form~\eqref{system}.
If the matrix $B_1$ satisfies that
\begin{itemize}
\item if $M<N$, the matrix $B_1$ is invertible,
\item if $M=N$, the matrices $B_1+la\Id$ are invertible for $l\geq 2$,
\item if $M>N$, no conditions are needed for $B_1$,
\end{itemize}
then there exist formal
power series $\hat R(t)=\sum_{n\geq 1} R_n t^n \in \R[[t]]$, $\hat
K(t) = \sum_{n\geq 0} K_n t^n \in \R[[t]]^{1+d+d'}$ with $K_0=
(0,0,0)$ and $K_1= (1,0,0)^\top$ such that
\begin{equation}\label{formal_inv_cond}
F\comp \hat K = \hat K\comp \hat R
\end{equation}
(in the sense of formal series composition).

More precisely, under these conditions, there exists a unique polynomial $R(t)= t -
a t^N + b t^{2N-1}$ such that for any $c\in \mathbb{R}$, there is a
unique formal power series $ \hat K(t) = \sum_{n\geq 1} K_n t^n \in
\R[[t]]^{1+d+d'}$ with $K_0= (0,0,0)^\top$, $K_1= (1,0,0)^\top$ and
$K_{N}^x=c$, satisfying~\eqref{formal_inv_cond}.

This expansion is $\gamma$-Gevrey, that is, there exist constants
$c_1,c_2>0$ such that
\[
\norm{K_n} \leq c_1 c_2^n n!^{\gamma}, \qquad n\ge 0,
\]
where $\norm{\cdot}$ is a norm in $\R^{1+d+d'}$.
\end{theorem}
We prove Theorem~\ref{formal_theorem} along
Sections~\ref{sec:formal} and~\ref{sec:Gevrey}. First, in
Proposition~\ref{construction} we prove the existence of the formal
solution of~\eqref{formal_inv_cond} and provide formulas to compute
it. Then, with the aid of some technical lemmas, we prove in
Proposition~\ref{gevprop} that this formal solution is
$\gamma$-Gevrey.

The following proposition emphasizes the conditions on our map given
by~\eqref{system} are not too restrictive when considering
parabolic-hyperbolic fixed points.

\begin{proposition}\label{prop:FormaF}
Let $\mathcal{F}:\mathcal{U}\to \R\times \R^d \times \R^{d'}$ be a real analytic map of the form
$$
\mathcal{F}(x,y,z) = (x, y , C z)+ \mathcal{N}(x,y,z), \qquad \mathcal{N}(x,y,z) =\O(\|(x,y,z)\|^2 ),
$$
with $1\not \in \Spec C$, having an invariant curve associated to the origin of the form
$(y,z) = \varphi(x)$. Assume that there exist $N\ge 2$ and $a\neq 0$
such that
\begin{equation}\label{hipprop}
\mathcal{N}^x(x,\varphi(x))= - a x^N +\O(|x|^{N+1})
\end{equation}
and that $\varphi$ is $\mathcal{C}^r$ with $r\geq N$. Then, by means
of changes of variables and a blow up, $\mathcal{F}$ can be expressed
as in the form~\eqref{system} for some $M\ge 2$.
\end{proposition}
The proof of this result is elementary. We defer it to
Appendix~\ref{A1}.

The following result assures that, under additional conditions, the
formal expansion $\hat K$ given by Theorem~\ref{formal_theorem} is
the asymptotic series of a true solution of the invariance equation,
analytic in some domain with $0$ at its boundary.
\begin{theorem}\label{mtexistencesolution}
Let $F$ be a map of the form~\eqref{system}. Assume that $a>0$ and
\begin{itemize}
\item If $M\geq N$, the matrix $C$ satisfies $\Spec C \subset \{ z\in \C : |z|\geq 1\}\backslash\{1\}$.
\item If $M<N$, the matrix $C$ satisfies $\Spec C \subset \{ z\in \C : |z|>1\}$ and the matrix
$B_1$ is such that $\Spec B_1 \subset \{ z\in \C : \Re z >0\}$.
\end{itemize}
Then, for any $0<\beta<\alpha\pi$, there exist $\rho$ small enough and a real analytic function $K$ defined on the open
sector
\begin{equation}
\label{def:sector} S= S(\beta,\rho)=\{ t= r e^{{\rm i} \varphi} \in
\C\ : \ 0<r<\rho, \; |\varphi|< \beta/2\}
\end{equation}
such that $K$ is a solution of the invariance equation $F\circ K= K\circ R$.

Moreover, K is
$\gamma$-Gevrey asymptotic to the $\gamma$-Gevrey formal solution $\hat K$.
That is, for any $0<\bar \beta <\beta$ and $0<\bar \rho <\rho$, there exist constants $c_1,c_2$ such that, for any $n\in \mathbb{N}$,
$$
\left \| K(t) - \sum_{j=1}^{n-1} K_j t^j \right \| \leq c_1 c_2^n n!^{\gamma} |t|^n,
$$
for all $t\in \bar{S}_1:=\{t=r e^{{\rm i} \varphi} \in \C\ : \ 0<r\leq \bar \rho, \; |\varphi|\leq \bar  \beta/2\}$.

In particular, $K$ can be extended to a $\mathcal{C}^{\infty}$ function in $[0,\rho)$.
\end{theorem}
The proof of this theorem is given in Section~\ref{sec:existence}.

Now we give conditions that ensure that the manifold given by
Theorem~\ref{mtexistencesolution} is unique (in a suitable open
set).
\begin{theorem}
\label{thm:uniqueness} Under the same assumptions of
Theorem~\ref{mtexistencesolution}, if the matrices $C$ and $B_1$
satisfy that
$$
\Spec C \subset \{ z\in \C : |z|>1\}, \qquad \text{and} \qquad \Spec B_1 \subset \{ z\in \C : \Re  z >0\},
$$
there exists a unique right hand side branch of a curve in the center manifold which is a parabolic stable manifold to the origin.
That is, if we denote by $B(\rho)\subset \R^{1+d+d'}$ the open ball of radius $\rho$, the following local stable manifold
$$
W^s_\rho =\{ (x,y,z) \in B(\rho) \ : \ F^k (x,y,z) \in B(\rho)\cap \{x>0\} ,\; \text{for all }k\geq 0\}
$$
satisfies that $W^{s}_\rho = K([0,\rho))$.
\end{theorem}
This theorem is proven by using the same geometrical arguments
in~\cite{BH08}. We omit the proof.

\begin{remark}
In the last two theorems we have assumed $a>0$. Clearly, if $a<0$,
the map $F^{-1}$ has the form in~\eqref{system} substituting $a,
B_1, B_2$ and $C$ by $-a, -B_1, -B_2$ and $C^{-1}$ respectively.
Therefore, if $a<0$, we can apply (if the other conditions are
satisfied) Theorems~\ref{mtexistencesolution}
and~\ref{thm:uniqueness} to $F^{-1}$ obtaining a local unstable
parabolic invariant manifold.
\end{remark}

A straightforward consequence of Theorem~\ref{mtexistencesolution}
is the following.
\begin{corollary}\label{cor:Fatou}
If $a\neq 0$, there exists a unique constant $b$ such that the real analytic maps
$$
f(x)=x-ax^{N} + f_{\geq N+1}(x),\qquad R(x) = x-ax^{N} + bx^{2N-1}
$$
are conjugated in a domain $S(\beta,\rho)$, with $0<\beta<\alpha\pi$ and $\rho$ small, by means of an
analytic function $h:S(\beta,\rho)\to \C$ which is
$\alpha$-Gevrey asymptotic to a $\alpha$-Gevrey formal series at $0$.
\end{corollary}

In the next section we will provide examples and describe the
parabolic manifolds as graphs of functions. We remark that
\begin{remark}
Let $F$ be a map of the form~\eqref{system} satisfying the
hypotheses of Theorem~\ref{mtexistencesolution}. Then, the graph
invariance equation
\begin{equation}
\label{graph_transform} F^{y,z} (x, \Phi(x) )=  \Phi(F^x(x,
\Phi(x)))
\end{equation}
with the condition $\Phi(0)=0$, $D\Phi(0)=0$,
has a $\gamma$-Gevrey solution if and only if the invariance
equation $F\comp  K = K\comp R$ has a $\gamma$-Gevrey solution with
$K(t)=(t,0)+\O(t^2)$ and $R(t) = t-at^{N} +b t^{2N-1}$. It is unique if the hypotheses of
Theorem~\ref{thm:uniqueness} are satisfied.

Indeed, if $\Phi$ satisfies~\eqref{graph_transform}, then
$\tilde{K}(x)=(x,\Phi(x))$ and $\tilde R = F^x(x, \Phi(x))$ is a
solution of $F\comp K = K \comp R$. Let $h$ be the $\gamma$-Gevrey
conjugation provided by Corollary~\ref{cor:Fatou}. Then $R(t)=
h^{-1} \circ  \tilde R \circ h(t) = t-at^{N}+bt^{2N-1}$ and $ K=
\tilde{K}\circ h$ is the solution we are looking for. On the
contrary, if $F\circ K = K\circ R$, then $\Phi(x) = K^{y,z}(
(K^x)^{-1}(x))$ is solution of the graph invariance equation. Notice
that $K^x(t) = t+ O(t^2)$ is invertible around the origin and its
inverse is Gevrey.

The same happens at a formal level. In this case, only the
hypotheses of Theorem~\ref{formal_theorem} are required.
\end{remark}

The statements in this section provide upper bounds to the
coefficients of the formal solution of the invariance equation. In
the next section we will give examples that show that our results
are sharp but also examples that show that a map satisfying our
hypotheses can have an analytic invariant manifold. To provide
conditions that ensure the existence of lower bounds of the
coefficients remains an open problem. The following proposition
shows that these conditions cannot depend only on a finite number of
coefficients of the Taylor expansion of~$F$ at the origin.

\begin{proposition}
\label{prop:trickylowerbounds} Let $F$ be an analytic map of the
form~\eqref{system} satisfying the hypotheses of
Theorem~\ref{formal_theorem} and $\hat \varphi (x) = \sum_{k\ge 2}
\varphi_k x^k$ a formal  solution of the invariance
equation~\eqref{graph_transform}. For $p\ge 0$, let $\varphi_{\le
p} (x) = \sum_{2\le k \le p}
\varphi_k x^k$.

Then, for any $p\ge 2$, there exists an analytic map $G$ such that
\[
\|F(x,y,z) -G(x,y,z)\| = \O(\|(x,y,z)\|^{p+1})
\]
with $\text{graph}\, \varphi_{\le p}$ as  invariant manifold to the
origin. If $p \ge \max\{N,M\}$, $G$ satisfies the hypotheses of
Theorem~\ref{formal_theorem} and, consequently, $\text{graph}\,
\varphi_{\le p}$ is a parabolic invariant manifold.
\end{proposition}

We defer the proof of this proposition
to Appendix~\ref{proofofpropositionprop:trickylowerbounds}.

In the following section we consider some examples. It is often
easier to provide examples of maps arising from flows. The following
remark is straightforward, but allows us to apply our results
directly to flows.
\begin{remark}
Let $X(x,y,z,t)$ be a $T$-periodic field, $(x,y,z)\in \R\times \R^{d}\times \R^{d'}$ of the form
\begin{equation}\label{formX}
X(x,y,z,t)
=
\begin{pmatrix} - a x^N  +   f_N(x,y,z,t) + f_{\geq N+1} (x,y,z,t) \\
 x^{M-1} B_1 y+ x^{M-1} B_2 z+ g_{M}(x,y,z,t)+ g_{\geq M+1}(x,y,z,t) \\
D z  +  h_{\geq 2}(x,y,z,t).
\end{pmatrix}
\end{equation}
Assume that the functions $f_N, f_{\geq N+1}, g_M, g_{\geq M+1},
h_{\geq 2}$ satisfy the hypotheses in Section~\ref{sec:notation} for
all $t\in [0,T]$ and that $0\notin \Spec D$. Then, any stroboscopic
Poincar\'e map of $\dot{\xi}=X(\xi,t)$ has the form
in~\eqref{system} with the same $a,B_1, B_2$ and $C=e^{TD}$.
\end{remark}


\section{Examples}\label{sec:examples}
In this section we provide several examples. In particular we show
that, under the hypotheses of Theorem~\ref{mtexistencesolution}, the
parabolic manifold (and, consequently, the formal solution) is
indeed $\gamma$-Gevrey and not more regular, that is, it is not
$\gamma'$-Gevrey for $0\le \gamma' < \gamma$.

It is more convenient to work with differential equations and manifolds represented as
graphs. That is, for a given periodic in time system $X(x,y,z,t)$ of
the form~\eqref{formX}, we look for formal solutions $(y,z)=\hat
\Phi(x,t)$, depending periodically on $t$, of the invariance
equation:
\begin{equation}\label{inveqX}
X^{y,z}(x,\hat \Phi(x,t),t)= D_x \hat \Phi(x,t) \big (X^x(x,\hat \Phi(x,t),t)\big ) +\frac{\partial\hat \Phi}{\partial t} (x,t).
\end{equation}

It is often useful to use the following equivalent definition of a
$s$-Gevrey series (see~\cite{Balser94}): a formal series
$\sum_{n\geq 0} a_n z^n$ is $s$-Gevrey if there exist $c_1,c_2
>0$ such that $|a_n| \leq c_1 c_2^n \Gamma(1+ s n)$, for all $n\ge
0$.

\subsection{Some elementary examples}
The first one is a generalization of the ones in~\cite{BH08}. Here
we add the variables corresponding to the eigenvalue equal to~$1$
but still require the presence of the hyperbolic directions.

\begin{claim}\label{Claim1} Let $X(x,y,z)$  be the autonomous vector field
\begin{equation}\label{systemClaim1}
X(x,y,z)=\big (-a x^N, x^{M-1} B_1  y + g(x,y,z), Cz+x^{\ell}c\big
),  \qquad (x,y,z)\in \mathbb{R}\times \mathbb{R}^d\times
\mathbb{R}^{d'},
\end{equation}
with $c\neq 0$, $b\neq 0$, $C$ an invertible matrix and
$g=g_{M}+g_{\geq M+1}$ as in~\eqref{system}.

Assume that $a,B_1$ satisfy their corresponding conditions in
Theorem~\ref{formal_theorem}. Then
\begin{itemize}
\item If $M\geq N$ and $d'\ge 1$, for any $g(x,y,z)=\O(\|(x,y,z)\|^{M+1})$, the formal invariance equation~\eqref{inveqX}
has a $\gamma$-Gevrey solution which is not $\gamma'$-Gevrey for any
$0\le\gamma' < \gamma$.
\item If $M < N$, taking $g(x,y,z)=x^{\nu} b$ with $\nu\geq M+1$, the formal invariance equation~\eqref{inveqX}
has a solution $\hat \Phi(x) = \sum_{n\geq 1} \Phi_n x^{n}$ which is
$\gamma$-Gevrey and is not $\gamma'$-Gevrey for any $0\le\gamma' <
\gamma$.
\end{itemize}
\end{claim}
\begin{proof}
We introduce $\hat \Phi(x) = (\hat \varphi(x),\hat \psi(x))$. We
have that $\hat \psi(x)=\sum_{n\geq 2} \psi_n x^n$ satisfies
$$
-a \sum_{n\geq 2 } n \psi_n x^{n+N-1} = C\sum_{n\geq 2} \psi_n x^{n} + x^{\ell}c
$$
or, equivalently,
$$
\sum_{n\geq 2} \psi_n x^{n} = - x^{\ell} C^{-1} c -a C^{-1}\sum_{n\geq N+1} (n-N+1) x^n \psi_{n-N+1}.
$$
Therefore $\psi_{2},\cdots = \psi_{\ell-1}=0$, $\psi_{\ell}=-C^{-1}c $ and
\begin{equation}\label{eqformalpsi}
\psi_{n} = -a (n-N+1) C^{-1} \psi_{n-N+1},\qquad n\geq \ell +1.
\end{equation}
Then $\psi_{n}=0$ if $n\neq \ell +k(N-1)$ and
$$
\psi_{\ell + k(N-1)} = (-1)^{k+1} a^k \prod_{i=0}^{k-1}
(\ell+i(N-1))  C^{-k-1} c,
$$
which implies, since $N-1\geq 1$, that $\|\psi_{\ell+k(N-1)}\| \geq \|c\| \|C\|^{-1} (a
\|C\|^{-1})^{k} \Gamma(\ell+k)/\Gamma(\ell)$. Then, using that
$\Gamma(\ell +k(N-1) \alpha)=\Gamma(\ell +k)$ we conclude that the
formal series $\hat \psi$ is exactly of order $\alpha=1/(N-1)$.

If $M\geq N$, $\gamma=1/(N-1)$, then, $\hat \psi$ is exactly of the Gevrey order claimed in
Theorem~\ref{formal_theorem}. Therefore, no matter what is the Gevrey order of $\hat \varphi$, the
asymptotic series $\hat \Phi$ is $\gamma$-Gevrey.

Now we consider the case $M<N$. The invariance of the formal
solution $\hat \varphi(x) = \sum_{n\geq 2} \varphi_n x^n$ reads
$$
-a \sum_{n\geq 2} n \varphi_n x^{n+N-1} = b x^{\nu} + B_1 \sum_{n\geq 2} \varphi_n x^{n+M-1}.
$$
Since $M<N$ and $B_1$ is invertible, $\varphi_2=\cdots =\varphi_{\nu-M}=0$, $\varphi_{\nu -M+1}=-B_1^{-1} b$, and
$$
\varphi_{n} = -a (n+M-N) B_1^{-1} \varphi_{n-N+M}, \qquad n\geq \nu-M+1.
$$
In the same way as in~\eqref{eqformalpsi}, it follows that $\hat \varphi$ is
Gevrey of order $\gamma=1/(N-M)$.
\end{proof}


We emphasize  that, when $M<N$, the map defined
by~\eqref{systemClaim1} has a Gevrey formal solution of order
precisely $\gamma=1/(N-M)$ even if $d' = 0$, that is, even if $F$ is
tangent to identity, but the same claim (for this particular example) only holds for $M\geq N$ if
$d'\ge 1$. In the next subsection we will deal with the case  $M\geq
N$ and $d'=0$, which is the relevant one in the problems of
celestial mechanics we will consider in
Section~\ref{sec:celestial_mechanics}.

\subsection{The tangent to the identity case (\texorpdfstring{$d'=0$}{Lg}) when \texorpdfstring{$M\geq N$}{Lg}}
In this section we present a family of differential equations of the
form~\eqref{formX} having a formal solution of the invariance
equation~\eqref{inveqX}. We check that this formal solution is
precisely $\gamma$-Gevrey. Recall that in this case
$\gamma=1/(N-1)$.

The example we will consider will be given by a non autonomous time
periodic vector field. The reason is because if the vector field is
autonomous, the parabolic invariant manifold is analytic (when
written as a graph), as the following claim shows.

\begin{claim} \label{claim:autonomousflowsanalyticmanifolds} Assume $M\geq N$. Let $X$ be an analytic vector field
of the form
\begin{equation}\label{formXautonomous}
X(x,y)=\big (-ax^N + f(x,y),  B_1 x^{M-1}y + g(x,y)\big ),\qquad  (x,y) \in \mathbb{R}\times \mathbb{R}^d
\end{equation}
with $f=f_{N}+f_{\geq N+1}$, $g=g_{M}+g_{\geq M+1}$ as
in~\eqref{system}, $a\neq 0$ and $B_1$ satisfying the condition stated in Theorem~\ref{formal_theorem}.
Then, the invariance
equation~\eqref{inveqX} has a real analytic solution $\varphi:
B_\rho \subset \mathbb{C} \to \mathbb{C}^d$  tangent to the $x$-axis
at the origin.

As a consequence the real analytic maps
$$
F(x,y) = (x-ax^N + \tilde{f}(x,y), y+B_1x^{M-1} y + \tilde{g}(x,y)),
$$
with $a\neq 0$, which are the time $1$ map of systems
like~\eqref{formXautonomous}, have an analytic solution of the graph invariance equation~\eqref{graph_transform}.
\end{claim}
\begin{proof}
The one dimensional invariant manifold we are looking for is the
graph of a function $y=\varphi(x)$ satisfying the equation
\begin{equation}\label{graphautonomous}
\frac{dy}{dx}=\frac{1}{-ax^{N} + f(x,y)} \big [B_1x^{M-1} y + g(x,y)\big ].
\end{equation}
We introduce the new variable $u$ by $y=xu$. The system becomes
\begin{equation}\label{blowupautonomous}
\frac{du}{dx} = \frac{1}{-ax^N+ f(x,xu)}\big [(ax^{N-1}\Id   + B_1x^{M-1} +x^{-1} f(x,xu))u  + x^{-1} g(x,xu)\big ].
\end{equation}
We introduce the functions $\bar{f}(x,u):=f(x,xu)$ and $\bar{g}(x,u):=g(x,xu)$ and we notice that they satisfy $\bar f(x,xu)=|x|^{N}\O(\|(x,u)\|)$ and
$\bar{g}(x,u)=\O(|x|^{M+1}) + \O(|x|^{M}\|u\|^2)$.  In addition, since $f$
and $g$ are analytic functions at $(x,y)=(0,0)$ and $M\geq N$, so are
the functions $x^{-N+1}\bar{f}(x,u), x^{-N}\bar{f}(x,u)$ and $x^{-N}\bar{g}(x,u)$ at $(x,u)=(0,0)$. We rewrite~\eqref{blowupautonomous} as
\begin{equation}\label{autonomousanalytic}
\frac{du}{dx} = \frac{1}{-ax +x^{-N+1}\bar{f}(x,u)} \big [(a\Id   + B_1x^{M-N} +x^{-N} \bar f(x,u) )u  + x^{-N} \bar{g}(x,u)\big ].
\end{equation}
We consider now the system
\begin{equation}\label{fieldhyperbolic}
\dot{x}=-ax + x^{-N+1}\bar{f}(x,u),\qquad \dot{u}=(a\Id   + B_1 x^{M-N} +x^{-N} \bar f(x,u))u  + x^{-N} \bar{g}(x,u).
\end{equation}
The origin is a fixed point, having a single hyperbolic direction corresponding to the eigenvalue
$-a$. Indeed, when $M>N$, the linear part of the field in~\eqref{fieldhyperbolic} at $(x,u)=(0,0)$ is
$A=\text{diag}(-a, a \Id)$.
However, when $M=N$, this linear part is
$$
A=\left (\begin{array}{cc} -a & 0 \\  \partial_x [x^{-N} \bar{g}(0,0)] & a\Id +B_1\end{array}\right),
$$
which may be not diagonal.
Using the non-resonance condition $-a\notin \Spec(B_1 + la \Id)$ if
$M=N$ and that $a\neq 0$ if $M>N$, one deduces from the theory of nonresonant invariant manifolds (\cite{CabreFL03a})
that the corresponding to the eigenvalue $-a$ one-dimensional invariant manifold is the graph of a
real function $h$, analytic at $x=0$, which is a solution of~\eqref{autonomousanalytic}.
Let $h(x)=cx+\O(x^2)$ (the constant $c$ is $0$ if either $M>N$ or $\partial_x [x^{-N} \bar{g}(0,0)]=0$). Then $y=xh(x)=\O(x^2)$ is  a real analytic
solution of~\eqref{graphautonomous} tangent to the $x$ axis.
\end{proof}

\begin{claim}
\label{example:optimal_Gevrey_N_equal_M} Let $X$ be the
$2\pi$-periodic vector field defined by
$$
X(x,y,t)=\big (-ax^{N}, bx^{N-1}y + x^{N+1}\cos t\big ), \qquad (x,y)\in \mathbb{R}^2
$$
with $a,b>0$. The parabolic stable manifold has a formal Taylor expansion at $0$ which is Gevrey of order exactly $\gamma=1/(N-1)$.
\end{claim}
\begin{proof}
We first note that Theorem~\ref{mtexistencesolution}
and~\ref{thm:uniqueness}  assure the existence and uniqueness of the
parabolic stable manifold when $a,b>0$.

For any initial conditions $x_0,y_0, t_0$, the associated flow  is
given by
\begin{align*}
x(t)&=\frac{x_0}{\big (1+ a (N-1) (t-t_0) x_0^{N-1}\big )^{\alpha}} \\
y(t)&= (1+a(N-1)(t-t_0)x_0)^{\beta} \left (y_0 + \int_{t_0}^t \frac{x_0^{N+1} \cos t }{\big (1+a(N-1) (t-t_0) x_0^{N-1}\big )^{\alpha (N+1)+\beta}}\,dt\right ),
\end{align*}
where we have introduced $\beta=b/a$. Since we are looking for the
stable invariant manifold, we want the solution such that
$(x(t),y(t))\to (0,0)$ as $t\to \infty$. Hence, since $\beta>0$, we
need to impose
$$
y_0=- x_0^{N+1} \int_{t_0}^\infty \frac{1}{\big (1+a(N-1) (t-t_0) x_0^{N-1}\big )^{\alpha (N+1)+\beta}} \cos t \,dt.
$$
Therefore the stable invariant manifold is described by
$$
y=\varphi(x,t) = -x^{N+1} \int_{0}^{\infty} \frac{1}{\big (1+a(N-1) \tau x^{N-1}\big )^{\alpha (N+1)+\beta}} \cos (t+\tau) \,d\tau.
$$
Notice that $\varphi$ is $2\pi$-periodic with respect to $t$.
Now we will prove that the series of $\varphi$ at $x=0$ is Gevrey of order $\gamma=1/(N-1)$. First we introduce
$\sigma=\alpha(N+1)+\beta>1$ and we decompose
\begin{align*}
\varphi(x,t)  &= - x^{N+1} \int_{0}^{\infty}\frac{e^{i(t+\tau)} + e^{-i(t+\tau)} }{2\big (1+a(N-1) \tau x^{N-1}\big )^\sigma} \,d\tau \\
&= - \mu x^2 e^{it} \int_{0}^{\infty}\frac{e^{\mu i x^{-(N-1)}u}}{2(1+u)^{\sigma}}\, du - \mu x^2 e^{-it} \int_{0}^{\infty}\frac{e^{-\mu i x^{-(N-1)}u}}{2(1+u)^{\sigma}}\, du ,
\end{align*}
with $\mu^{-1} = a(N-1)$.
We take $\theta \in (0,\pi)$ and change the integration path in the above integrals as:
$$
\varphi(x,t) = -\mu x^2 e^{it} \int_{0}^{\infty e^{i\theta} }\frac{e^{\mu i x^{-(N-1)}u}}{2(1+u)^{\sigma}}\, du
- \mu x^2 e^{-it}\int_{0}^{\infty e^{-i\theta}}\frac{e^{-\mu i x^{-(N-1)}u}}{2(1+u)^{\sigma}}\, du .
$$
It is well known that these integrals define the confluent
hypergeometric functions $\Psi$ (see~\cite[p.\ 280]{Erdelyi53}) so that
$$
\varphi(x,t) = - \frac{\mu}{2}  x^{2} \left ( e^{it} \Psi\big (1,-\sigma+2, \mu i x^{-(N-1)}\big ) + e^{-it} \Psi\big (1,-\sigma+2, -\mu i x^{-(N-1)}\big )\right ).
$$
By~\cite[p.\ 302]{Erdelyi53}, an asymptotic expansion of
$\Psi(a,c,z)$ for large $|z|$ is
$$
\Psi(a,c,z) = \sum_{n\geq 0 } (-1)^n \frac{(a)_n (a-c+1)_n}{n!} z^{-a-n},
$$
with $(a)_n=\Gamma(a+n)/\Gamma(a)$. Therefore, the Taylor formal series at $0$ of $\varphi$ is
\begin{align*}
\hat \varphi(x,t) =& -\frac{\mu}{2} x^{2} \frac{1}{\Gamma(\sigma)} \sum_{n\geq 0} (-1)^n \Gamma(n+\sigma) \left (\big (\mu i x^{-(N-1)})^{-n-1} + \big (-\mu i x^{-(N-1)})^{-n-1}\right ) \\
=&\mu x^2 \frac{1}{\Gamma(\sigma)} \cos t \sum_{n\geq 0} \Gamma(2n+1+\sigma)  \frac{x^{(N-1)(2n+2)}}{\mu^{2n+2}} \\
&-\mu x^2 \frac{1}{\Gamma(\sigma)} \sin t \sum_{n\geq 0} \Gamma(2n+\sigma)  \frac{x^{(N-1)(2n+1)}}{\mu^{2n+1}}.
\end{align*}
This formal series is Gevrey of order $\gamma=1/(N-1)$. Indeed,
comparing $\Gamma(k+\sigma)$ with $\Gamma(1+
\gamma(N-1)(k+1))=\Gamma(k+2)$  we conclude that $\hat \varphi$ is
a Gevrey formal series of order exactly $\gamma$.
\end{proof}

\subsection{Aplications to Celestial Mechanics}
\label{sec:celestial_mechanics}

The three body problem describes the motion of three point bodies
evolving under their mutual Newtonian gravitational attraction. The
restricted three body problem is the simplification of the three
body problem obtained by assuming that one of the bodies has zero
mass. Consequently, the bodies with mass, usually called
\emph{primaries}, describe Keplerian orbits. See, for instance,
\cite{MeyerH92}.

Among the several instances of the restricted three body problem one
finds the Sitnikov problem, which is the special case when the
primaries move in ellipses and the massless body in the line
orthogonal to the plane of the primaries through their center of
mass. The relevant parameter in the Sitnikov problem is the
eccentricity~$e$ of the orbits of the primaries. When $e=0$, the
Sitnikov problem is integrable. Another important subproblem is the
so called restricted planar three body problem (RPTBP), when the
massless body moves in the plane where the primaries lie, while the
latter describe Keplerian ellipses. In this case, a relevant
parameter is the mass ratio of the primeries, $\mu$, which can be
assumed to be in $[0,1/2]$. When $\mu =0$, the RPTBP is integrable.

In both cases, the parabolic infinity can be written as
\begin{equation}
\label{eq:system_in_McGehee_coordinates}
 \begin{aligned}
 \dot x&= -\frac{1}{4} (x+y_1)^3(x+(x+y_1)^3\O_0)\\
 \dot y_1&= \frac{1}{4} (x+y_1)^3(y_1+(x+y_1)^3\O_0)\\
 \dot{{\tilde y}}&=\frac{1}{4}(x+y_1)^2(x-y_1) {\tilde y} + (x+y_1)^5\O_0\\
 \dot t &=1
 \end{aligned}
\end{equation}
where $(x,y_1,\tilde y)\in \R\times \R\times \R^n$ and $\O_k$ stands
for a function in $(x,y_1,\tilde y,t)$, $1$-periodic with respect
to~$t$, analytic in a neighborhood of $x=y_1=0$, $\tilde y = 0$ and
of order $\O(\|(x,y_1,\tilde y)\|^k)$. In the case of the Sitnikov
problem, $n=0$, while $n=2$ in the RPTBP. See~\cite{Moser01} for the
derivation of the above equations in Sitnikov problem
and~\cite{GuardiaMSS17} in the planar restricted three body problem.

It is immediate to check that any stroboscopic Poincar\'e map of the
system~\eqref{eq:system_in_McGehee_coordinates} has the form
\[
\begin{pmatrix}
q \\ p \\ z
\end{pmatrix}
\mapsto
\begin{pmatrix}
x -\frac{1}{4} x^4 +y_1 \O_3+ \O_5\\
y_1  + \frac{1}{4}y_1 x^3 + y_1^2 \O_2 +\O_5\\
{\tilde y}  + \frac{1}{4}{\tilde y} x^3 + y_1 {\tilde y} \O_2 +\O_5\\
\end{pmatrix},
\]
which has the form~\eqref{system} with
\[
d=1+n, \quad d'=0, \quad N = M = 4, \quad a = \frac{1}{4}, \quad B_1
= \frac{1}{4}\Id_{1+n}.
\]
Consequently $\alpha = 1/3$. Since the eigenvalues of $B_1$ are
positive, Theorems~\ref{formal_theorem}, \ref{mtexistencesolution}
and~\ref{thm:uniqueness} apply. Hence we have

\begin{corollary}
The parabolic infinity in the Sitnikov problem (for any $e\in
[0,1)$) and in the RPTBP (for any $\mu \in [0,1/2])$ possesses
invariant manifolds which are $1/3$-Gevrey.
\end{corollary}

As we have already mentioned, Theorems~\ref{formal_theorem},
\ref{mtexistencesolution} only provide upper bounds on the
coefficients of the expansion of the invarariant manifold. However,
in view of Mart\'{\i}nez and Sim\'o's numerical computations \cite{MartinezS14} and
the example in Claim~\ref{example:optimal_Gevrey_N_equal_M}, where a
time periodic perturbation of a system with a parabolic fixed point
is considered, we present the following conjecture.

\begin{conjecture}
\label{conjecture} The parabolic infinity in the Sitnikov problem,
with $e\in (0,1)$, and in the RPTBP, with $\mu \in (0,1/2]$,
possesses invariant manifolds which are \emph{precisely}
$1/3$-Gevrey, that is, they are not $\gamma'$-Gevrey for any $0\le
\gamma' < 1/3$.
\end{conjecture}

\section{Formal parameterization of the manifold}\label{sec:formal}
In this section we obtain a formal solution of the equation $F\comp
K = K\comp R$, that is, a formal series which solves the equation at
all orders. We will need also a precise expression of the
coefficients in order to obtain Gevrey estimates for them.

We will use the following notation, that arises from the
Fa\`a-di-Bruno formula. Assuming that $f$ and $g$ are two
$\mathcal{C}^{\infty}$ functions such that $f \circ g$ makes sense,
$f(0)=0$ and $g(0)=0$, we have that  $(f\comp g)_l=\frac{1}{l!}
D^l(f\comp g)(0)$ satisfies
\begin{equation}\label{Faa_power}
(f\circ g)_l =
\sum_{k=1}^{l} \sum_{\begin{array}{c} \scriptstyle{l_1+\cdots+l_k=l} \\
\scriptstyle{l_i\ge 1} \end{array}}
 f_k [g_{l_1}, \cdots , g_{l_k}].
\end{equation}
Here $f_k$ and $g_k$ are $k$-multilinear symmetric maps. This
expression also holds when dealing  with the composition of formal
power series $\hat f(w)= \sum_{l\geq 1} f_l w^l$ and $\hat g(v)=
\sum_{l\geq 1} g_l v^l$. The coefficient of the~$l$ order term of
the formal composition $f\comp g$ is given by~\eqref{Faa_power}. It
depends only on $ f_{\leq l}(w) = \sum_{k= 1}^l f_k w^k$ and $
g_{\leq l}(v) = \sum_{k= 1}^l g_k v^k$. The only term of $(\hat
f\comp \hat g)_l$ in which $f_l$ appears is $f_l g_1^l$, and the
only term in which $g_l$ appears is $f_1 g_l$.

We introduce the maps
\begin{equation}
\label{defL}
\L(x,y,z)=\left (\begin{array}{c} x-a  x^N \\ y
\\ C z \end{array}\right ),\qquad G(x,y,z)= F(x,y,z)-\L(x,y,z),
\end{equation}
for $l\geq 2$, the family of operators
\begin{equation}\label{defAl}
\A_l=\begin{cases}
-(B_1+al\Id)^{-1},  &\quad \text{if}\;\;\; N=M, \\ -B_1^{-1},  & \quad \text{if}\;\;\; M<N, \\
-(la)^{-1}\Id, &\quad \text{if}\;\;\; M>N
\end{cases}
\end{equation}
and
$$
L=\min\{N,M\}.
$$

\begin{proposition}\label{construction}
There exists a unique $b\in \R$ such that for any $c\in\R$ there
exists a unique formal power series $\hat{K}(t)= \sum_{l=1}^{\infty}
K_l t^l $, $K_l\in \R^{1+d+d'}$ with $K_1=(1,0,0)^\top$ and $K^x_N=c$,
such that $R(t)=t-a t^N +b t^{2N-1}$ and ${K}$ satisfies
 the equation $F\comp K -  K\comp  R=0$ formally.

The coefficients of $K$ and $R$ can be given
inductively. For $l>1$ we have
\begin{align*}
K_l^z & = -(C-\Id)^{-1} E_{l}^z, \\
K_l^y & =\begin{cases}
\displaystyle  \A_l E_{l+L-1}^y, &\text{if}\quad  N< M,\\
\displaystyle \A_l (E_{l+L-1}^y+B_2 K_l^z), &\text{if}\quad  N\ge M,
\end{cases}
\\
K_l^x & =\begin{cases} \displaystyle  \frac{-1}{a(l-N)} (E_{l+N-1}^x
+ v^\top K_l^y+ w^\top K_l^z),
&\text{if}\quad  l\neq N,\\
\displaystyle c, &\text{if}\quad  l= N,
\end{cases} \\
\intertext{and}\\
R_{l+N-1}&  =\begin{cases}
\displaystyle  0, &\text{if}\quad l>  1, \;\; l\ne N,\\
\displaystyle b= E_{2N-1}^x + v^\top K_N^y+ w^\top K_l^z,
&\text{if}\quad  l= N;
\end{cases}
\end{align*}
where
\begin{align}
E_{l+N-1}^x = &
 -a \sum_{\begin{array}{c}\scriptstyle{l_1+\cdots+l_N=l+N-1} \nonumber \\
\scriptstyle{1\leq l_i \leq l-1}
\end{array}} \prod_{i=1}^N K^x_{l_i} +
\sum_{k=N}^{l+N-1}
\sum_{\begin{array}{c} \scriptstyle{l_1+\cdots+l_k=l+N-1} \\
\scriptstyle{1\leq l_i \leq l-1}
\end{array}} G^x_k[K_{l_1},\cdots, K_{l_k}] \label{Ex} \\
& - \sum_{k=2}^{l-1} K^x_k
\sum_{\begin{array}{c} \scriptstyle{l_1+\cdots+l_k=l+N-1} \\
\scriptstyle{1\le l_i \le l+N-2} \end{array}} \prod_{i=1}^k R_{l_i},
 \\
E_{l+L-1}^y = &
  \sum_{k=M}^{l+L-1}
\sum_{\begin{array}{c} \scriptstyle{l_1+\cdots+l_k=l+L-1} \\
\scriptstyle{1\leq l_i \leq \min\{l-1,l+L-M\}}
\end{array}} G^y_k[K_{l_1},\cdots,K_{l_k}] \nonumber
\\ &-
\sum_{k=M-L+2}^{\min\{l-1, l+L-N\}}K_k^y \sum_{\begin{array}{c}
\scriptstyle{l_1+\cdots+l_k=l+L-1} \\ \scriptstyle{1\le l_i \le
l+N-2}
\end{array}} \prod_{i=1}^k R_{l_i}
\label{Ey} \\
\intertext{and} \nonumber \\
\label{Ez} E_{l}^z = &  \sum_{k=2}^{l}
\sum_{\begin{array}{c} \scriptstyle{l_1+\cdots+l_k=l} \\
\scriptstyle{1\leq l_i \le l-1}
\end{array}} G^z_k[K_{l_1},\cdots,K_{l_k}]
- \sum_{k=2}^{l-N+1}K_k^z \sum_{\begin{array}{c}
\scriptstyle{l_1+\cdots+l_k=l} \\ \scriptstyle{1\le l_i \le l+N-2}
\end{array}} \prod_{i=1}^k R_{l_i}.
\end{align}
%
%

In addition, if $1\le l\leq M-L+1$, $K^y_l=0$.
\end{proposition}
\begin{proof}
First we prove by induction that there exist a formal series
$K=\sum _{n\ge 1}K_n t^n$, $K_n =(K^x_n, K^y_n, K^z_n)^\top \in \R^{1+d+d'}$ and a polynomial
$R(t) = \sum _{n\ge 1}^{n_0} R_n t^n$, $R_n  \in \R$, with as much as possible coefficients equal to 0, such that the error
$$
E^l = F \circ K_{\le l} - K_{\le l} \circ R_{\le l+N-1}
$$
satisfies
\begin{equation}
\label{eq:orderEl} E^{l}(t) =
\begin{cases}
\big (\O(t^{l+N}) ,\O(t^{M+1}),\O(t^{l+1})\big )^{\top}, & \quad \text{if}
\;\;\; l < M-L+1, \\
\big (\O(t^{l+N}) ,\O(t^{l+L}),\O(t^{l+1})\big )^{\top}, & \quad \text{if} \;\;\;
l \ge M-L+1. \end{cases}
\end{equation}
To deal simultaneously with both cases we introduce $P(l)$ as
\begin{equation*}
P(l) = \begin{cases} M+1, &\quad \text{if} \;\;\; 1\le l < M-L+1, \\
l+L,& \quad l \ge M-L+1.
\end{cases}
\end{equation*}
Note that $P(l-1) + 1 \ge P(l) $ and that $P(l) = \max \{M+1,l+L\}$.

We can write $E^l (t)= \sum_{n\ge 1} E^l_n t^n$, with $E^l_n \in \mathbb{R} \times \mathbb{R}^d \times \mathbb{R}^{d'} $.
We denote by $E^{l,x}_{l+N-1}, E^{l,y}_{P(l)}$ and $E^{l,z}_l$ the first non-zero terms of
$(E^{l,x}, E^{l,y}, E^{l,z})$ respectively.
From the proof it will become clear that $E^{m,x}_{l+N-1}, E^{m,y}_{P(l)}$ and $E^{m,z}_l$
actually do not depend on $m$ provided  $m\ge l-1$.
We will simply denote them by $E^{x}_{l+N-1}, E^{y}_{P(l)}$ and $E^{z}_l$ respectively.
These values are the ones which
appear in the statement.

Taking $R(t) = t- a t^N + \O(t^{N+1}) $ and $K_1 = (1,0,0)^{\top}$
the claim holds true for $l=1$ because
\[
E^1 (t)= F \circ K_{\le 1}(t) - K_{\le 1} \circ R_{\le N} (t)= (\O(t^{N+1}),
\O(t^{M+1}), \O(t^2)).
\]

Now, let $l\ge 2$ and assume that there exist polynomials $K_{\le l-1}$ of degree at most $l-1$
and $R_{\le l+N-2} $ of degree at most $l+N-2$ such that
\[
E^{l-1}(t) =
\big (\O(t^{l+N-1}) ,\O(t^{P(l-1)}),\O(t^{l})\big )^{\top}.
\]
We remark that the value of the constant  $b=R_{2N-1}$ will be determined at the step
$l=N$.

In addition, we assume that $K_j^y= 0$ $1 \le j \le l-1 \le M-L+1$.

We look for $K_l \in \R \times \R^{d} \times \R^{d'}$ and $R_{l+N-1}\in \R$
such that $K_{\le l}(t) = K_{\le l-1}(t)
+K_l t^{l}$  and $R_{\le l+N-1}(t) = R_{\le l+N-2}(t) + R_{l+N-1} t^{ l+N-1}$ satisfy~\eqref{eq:orderEl}. Defining  $\Delta_l(t) =
K_l t^{l}$, we have that
%
\begin{align}
E^l=F \circ & K_{\le l} -  K_{\le l} \circ R_{\le l+N-1}\nonumber \\
= &  E^{l-1} \nonumber \\
& + F \circ K_{\le l} - F \circ K_{\le l-1} - (DF\circ K_{\le l-1}) \Delta_l \label{eq:steplofinduction1} \\
& + (DF\circ K_{\le l-1}) \Delta_l \label{eq:steplofinduction2}\\
& - K_{\le l} \circ R_{\le l+N-1} - K_{\le l} \circ R_{\le l+N-2} \label{eq:steplofinduction3}\\
& - \Delta_l \circ R_{\le l+N-2}. \label{eq:steplofinduction4}
\end{align}
By the induction hypothesis,
\begin{equation*}
E^{l-1}(t) = (E_{l+N-1}^x t^{l+N-1}, E_{P(l-1)}^y
t^{P(l-1)},E_{l}^z t^{l})^{\top} +\big (\O(t^{l+N})
,\O(t^{P(l)}),\O(t^{l+1})\big )^{\top}.
\end{equation*}

Now we identify the lowest order terms
in~\eqref{eq:steplofinduction1}, \eqref{eq:steplofinduction2},
\eqref{eq:steplofinduction3} and~\eqref{eq:steplofinduction4}.

Using that  $l\ge 2$ we easily estimate \eqref{eq:steplofinduction1}
\begin{equation*}
\begin{aligned}
(F \circ K_{\le l} - F \circ K_{\le l-1} - (DF\circ K_{\le l-1})
\Delta_l)(t) & = \big (\O(t^{2l+N-2}) ,\O(t^{2l+M-2}),\O(t^{2l})\big
)^{\top} \\
& = \big (\O(t^{l+N}) ,\O(t^{P(l-1) + 1}),\O(t^{l+1})\big )^{\top}
\end{aligned}
\end{equation*}
since $2l+M-2\ge M+1+2l-3 > M+1$ and $2l+M-2 \ge l+L +l-2 \ge l+L$.

Concerning \eqref{eq:steplofinduction2}, taking into account that $K_1 = (1,0,0)^{\top}$,
$$
((DF\circ K_{\le l-1}) \Delta_l)(t) = \begin{pmatrix} (1-Nat^{N-1})t^{l}K_l^x
+ v^{\top} K_l^y t^{l+N-1}+ w^{\top} K_l^z t^{l+N-1} + \O(t^{l+N})\\
K^y_l t^l+B_1K_l^y t^{l+M-1} + B_2 K_l^z t^{l+M-1}+ \O(t^{l+M}) \\
C K_l^z t^l + \O(t^{l+1})
\end{pmatrix}.
$$
As for \eqref{eq:steplofinduction3}, taking into account that $K_j^y = 0$ if $1 \le j \le l-1 \le M-L+1$, which implies that
\[
K_{\le l} (t) =
\begin{cases}
K_l^y t^l, & l-1 \le M-L+1, \\
K_{M-L+1} t^{M-L+1} + O(t^{M-L+2}), & \text{otherwise},
\end{cases}
\]
we have that
\begin{equation*}
(K_{\le l} \circ R_{\le l+N-1} - K_{\le l} \circ R_{\le l+N-2})(t) =
\begin{pmatrix}
R_{l+N-1} t^{l+N-1} \\ 0 \\0
\end{pmatrix}
+
\begin{pmatrix}
\O(t^{l+N}) \\ \Delta_l^y(t) \\ \O(t^{l+N})
\end{pmatrix},
\end{equation*}
where
\[
\Delta_l^y(t) =
\begin{cases}
l K_l^y O(t^{2l+N-2}), & l-1 \le M-L+1, \\
O(t^{l+M+N-L}), & \text{otherwise},
\end{cases}
\]
In all cases $\O(t^{l+M+N-L}) = \O(t^{P(l-1)+1})$.

Finally we evaluate \eqref{eq:steplofinduction4}
\begin{equation*}
\Delta_l \circ R_{\le l+N-2} (t) =  K_l(t-at^N +\O(t^{N+1}))^l  = K_l
t^l - l a K_l t^{l+N-1}+ \O(t^{l+N})K_l.
\end{equation*}

From the above calculations, since $l\ge 2$, we have
\begin{align*}
E^l(t) =&
\begin{pmatrix} (E_{l+N-1}^x+a(l-N)K_l^x
- R_{l+N-1}+v^{\top} K_l^y + w^{\top} K_l^z) t^{l+N-1} \\
E_{P(l-1)}^y t^{P(l-1)}+al K_l^y t^{l+N-1}+ B_1 K_l^y t^{l+M-1} + B_2 K_l^z t^{l+M-1} \\
(E_l^z + (C-\Id) K_l^z) t^l
\end{pmatrix}
\\
&+\begin{pmatrix}  \O(t^{l+N})\\
 \O(t^{P(l-1)+1})+\O(t^{l+N})K_l^y \\
 \O(t^{l+1})
\end{pmatrix}.
\end{align*}
This expression permits to choose $(K_l^x,K_l^y,K_l^z)$ and $R_{l+N-1}$ in order to $E^l$ has the claimed order. We start dealing with the third component. We have to take
\[
K_l^z  = -(C-\Id)^{-1} E_l^z.
\]
When dealing with the second one we have to distinguish two cases.
If $P(l-1)=M+1$, which means that $l< M-L+1$ we have that $l+L-1< M$.
Then if $L=N<M$, $K_l^y=0$. Otherwise, if $L=M\le N$, $l\le 1$ and this case is void. Now suppose
$P(l-1)=l+L-1$.
If $M\le N$, we take
\[
K_l^y = \A_l (E_{l+L-1}^y+B_2 K_l^z),
\]
while, if $N<M$,
\[
K_l^y = \A_l E_{l+L-1}^y.
\]
Finally, considering the $x$
component, if $l\neq N$, we take
\[
R_{l+N-1} = 0, \qquad K_l^x = -\frac{1}{a(l-N)}\big ( E_{l+N-1}^x
+v^{\top} K_l^y + w^{\top} K_l^z\big),
\]
otherwise,
\[
R_{2N-1} = E_{2N-1}^x+v^{\top} K_N^y + w^{\top} K_N^z, \qquad  K_N^x \;\;\; \text{ is free }.
\]
We write $c=K_N^x$ which can be chosen arbitrarily. We recall that
$R_{2N-1}$ corresponds to the coefficient $b$.

Now we come to compute $E^{x}_{l+N-1}$, $E^y_{l+L-1}$  and
$E^z_{l+L-1}$.

By definition, $E^z_{l}$ is the term of order $l$ of $\pi^z E^{l-1}
= F^z \circ K_{\le l-1} - K^z_{\le  l-1} \circ R_{\le  l+N-2}$, that is,
\[
E^z_{l} = \frac{1}{l!} D^l \pi^z E^{l-1}(0).
\]
By the Fa\`a di Bruno formula  \eqref{Faa_power},
\begin{equation}
\label{eq:Ezl} E^z_{l} = \sum_{k=1}^{l} \sum_{\begin{array}{c}
\scriptstyle{l_1+\cdots+l_k=l} \\ \scriptstyle{1\leq l_i \leq l-1}
\end{array}} F_k^z[K_{l_1},\dots,K_{l_k}] - \sum_{k=1}^{l-1} K_k^z \sum_{\begin{array}{c}
\scriptstyle{l_1+\cdots+l_k=l}\\ \scriptstyle{1\leq l_i \leq l+N-2}
\end{array}} \prod_{i=1}^k R_{l_i}.
\end{equation}

In the first term of \eqref{eq:Ezl} the addend with $k=1$ vanishes because $K^z_1=0$.
Moreover, for $k\ge 2$, $F^z_k = G^z_k$. In the second term, the addend with $k=1$ also
vanishes. Moreover, if $k>l-N+1$, since
$$
\sum_{\begin{array}{c} \scriptstyle{l_1+\cdots+l_k=l}\\
\scriptstyle{1\leq l_i \leq l+N-2}
\end{array}} \prod_{i=1}^k R_{l_i}
$$
is the coefficient of $t^l$ in
\[
R(t)^k = (t-a t^N + b t^{2N-1})^k = t^k - a k t^{k+N-1} + \O(t^{k+2N-2}),
\]
the addend with this $k$ vanishes because then $l<k+N-1$ and the next non-zero
term after order $k$ is of order $k+N-1$.
This proves formula~\eqref{Ez}.

Analogously,
$$
E_x^{l+N-1} = \frac{1}{(l+N-1)!} \Dif^{l+N-1} \pi^x E^{l-1}(0),\qquad
E_y^{l+L-1} = \frac{1}{(l+L-1)!} \Dif^{l+L-1} \pi^y E^{l-1}(0).
$$
Applying again Fa\`a di Bruno's formula  we obtain
\begin{equation}
\label{ElL1y} E_{l+L-1}^y = \sum_{k=1}^{l+L-1}
\sum_{\begin{array}{c} \scriptstyle{l_1+\cdots+l_k=l+L-1} \\
\scriptstyle{1\leq l_i \leq l-1}
\end{array}} F^y_k[K_{l_1},\cdots,K_{l_k}]
-
\sum_{k=1}^{l-1}K_k^y
\sum_{\begin{array}{c}
\scriptstyle{l_1+\cdots+l_k=l+L-1} \\ \scriptstyle{1\leq l_i \leq l+N-2}
\end{array}} \prod_{i=1}^k R_{l_i}
\end{equation}
and
\begin{equation*}
E_{l+N-1}^x = \sum_{k=1}^{l+N-1}
\sum_{\begin{array}{c} \scriptstyle{l_1+\cdots+l_k=l+N-1} \\
\scriptstyle{1\leq l_i \leq l-1}
\end{array}} F^x_k[K_{l_1},\cdots, K_{l_k}] - \sum_{k=1}^{l-1} K^x_k
\sum_{\begin{array}{c} \scriptstyle{l_1+\cdots+l_k=l+N-1} \\
\scriptstyle{1\leq l_i \leq l+N-2} \end{array}}
\prod_{i=1}^k R_{l_i}.
\end{equation*}
We begin by determining the indices in \eqref{ElL1y} that provide non-zero terms in $E_{l+L-1}^y$.
The term with $k=1$ in the first addend \eqref{ElL1y} vanishes because
it would be $F^y_1 K^y_{l+L-1}$, but for $E^l$ we are working with $K_{\le l-1}$.
Moreover, since $F_k^y=0$ if $2\leq k\leq M-1$, the sum must start with $k=M$. Also,
$
M\leq k \le l_1+\cdots+l_k =l+L-1
$
implies $l\geq M-L+1$.
In addition, when $l=M-L+1$, we always  have that
$$
l_1+\cdots+l_k=M-L+1+L-1=M.
$$
Therefore, if $l=M-L+1$, $k=M$ and $l_1=\cdots=l_M=1$. Since $K_1=(1,0,0)^{\top}$, the corresponding term is
$$
F_M^y [\overbrace{K_1,\cdots,K_1}^{M)}]= \frac{1}{M!}\partial_x^M F^y (0,0)=0.
$$
Then if $l\leq M-L+1$ the first term is void. To finish with the
first term, we note that for all $i$, using again that $k\geq M$,
$$
M-1+l_i \leq l_1+\cdots+l_k =l+L-1\quad  \Rightarrow \quad l_i \leq l+L-M,
$$
that is, the first term of~\eqref{ElL1y} has the form claimed in
formula~\eqref{Ey}. With respect to the second term
of~\eqref{ElL1y}, we only need to note that $K^y_k=0$ for $1\le k\le M-L+1$, and analogously as before, that
$$
\sum_{\begin{array}{c} \scriptstyle{l_1+\cdots+l_k=l+N-1} \\
\scriptstyle{1\leq l_i \leq l+N-2} \end{array}} \prod_{i=1}^k
R_{l_i}
$$
is the coefficient of $t^{l+N-1}$ of $R(t)^k$. Therefore since if vanishes for
$k<l+L-1<k+N-1$, we have that $l+L-1\geq
k+N-1$ which implies that  $k\leq l+L-N$. This ends the proof of formula~\eqref{Ey}
for $E_l^y$. To check formula~\eqref{Ex} for $E_{l+N-1}^x$ we use the form
of $F^x(x,y,z)=x-ax^N + G^x(x,y,z)$ and the proof follows the same lines
as the one for $E_{l+L-1}^y$.
\end{proof}

\section{Gevrey estimates}\label{sec:Gevrey}
Before starting to obtain the Gevrey estimates of the formal
solution $K$ we perform two change of coordinates. The first one is
a close to the identity change that uses the $(N-1)$-degree
approximation of the formal parabolic curve obtained in Proposition
\ref{construction} to put it closer to the $x$-axis. In the new
variables the parameterization will be the embedding to the $x$-axis
plus terms of order at least $N$.

The structure of this section is quite similar to the counterpart in~\cite{BH08}, however,
there are some differences to take into account.

\begin{lemma}\label{lemmaN}
We define the change of variables
\begin{equation*}
(x,y,z)=\Phi(\bar{x},\bar{y},\bar{z}):=K_{\leq N-1}(\bar{x})+(0,\bar{y}, \bar{z}),
\end{equation*}
where $K_{\leq N-1}(\bar{x})=\sum_{j=1}^{N-1}K_{j}\bar{x}^j$. In these new
variables:
\begin{enumerate}
\item $\bar{F}=\Phi^{-1} \circ F \circ \Phi$ has the same form~\eqref{system}
of~$F$ with the same constant $a$, vectors $v^{\top}, w^{\top}$ and
the same matrices $B_1$ and $C$.
\item The formal solution $\bar{K}$ and $\bar R$ of
$\bar{F}\circ \bar{K} -\bar{K}\circ \bar R=0$ obtained applying
Proposition~\ref{construction} to $\bar{F}$ satisfies
\begin{equation*}
\bar{K}(t)=(t,0,0)^{\top}+  O(t^{N})\qquad \text{and}\qquad \bar
R(t)=t-at^N + b t^{2N-1} = R(t).
\end{equation*}
\item The Gevrey character is not affected by this change, i.e., if one of $K$ or $\bar{K}$ is Gevrey of some order
the other is also Gevrey of the same order.
\end{enumerate}
\end{lemma}
The proof of this lemma depends on cumbersome but straightforward computations and uses, among other properties, that
$K_1^y=\cdots=K_{M-L+1}^y=0$.

Next we perform a rescaling of parameter $\lambda$ to achieve a good
control on the growth of the terms $K_l$ of the formal solution up
to some suitable order $l_0$ so that  we can start an induction
procedure to estimate the terms $K_j$ from $l_0$ on and obtain a
significantly simpler bound from them.

Let $\U \subset \C^{1+d+d'}$ be the domain of a complex extension of $\bar F$.
Let $ B(\delta)$ be a ball of radius $\delta >0$ such that
$\overline{ B(\delta)} \subset \U$.

Let $\bar G = \bar F - \L$ with $\L$ defined in~\eqref{defL}. Given $\esc\geq 1$ we introduce
\begin{equation}\label{deftildes}
\begin{aligned}
\tilde{F}(x,y,z) &= \esc \bar{F} (\esc^{-1}x ,\esc^{-1} y,\esc^{-1}z),  &\qquad  \tilde{G}(x,y,z) &= \esc \bar{G} (\esc^{-1}x ,\esc^{-1} y,\esc^{-1}z), \\
\tilde{K}(t) &= \esc \bar{K} (\esc^{-1}t),&\qquad \tilde{R}(t) &= \esc \bar{R}(\esc^{-1} t),
\end{aligned}
\end{equation}
and for the sake of simplicity, we omit the dependence of $\esc$ on the notation.
\begin{lemma}\label{scaling}
Let $\bar\delta$ be such that $\overline B(\bar\delta)$ is contained
in the complex domain $\mathcal U$ of $\bar F$.
Let $\bar G= \bar F - \L$ and
$\bar \CC= \max_{(x,y,z)\in \overline B(\bar\delta)} \norm{\bar{G}(x,y,z)}$.

For all  $\delta_0, \e, \mu_0>0$ and
$l_0\in \N$, there exists $\esc:=\esc(\delta_0, \e, \mu_0, l_0, \bar\delta)\geq 1$
such that the functions $\tilde{F}$, $\tilde{G}$, $\tilde{K}$ and $\tilde{R}$ defined above in~\eqref{deftildes},
satisfy the following properties:
\begin{enumerate}
\item $\tilde{F}$ has the form (\ref{system}), where the corresponding
values of $a,v,B_1,B_2$ and $C$ (which we denote with the same letter with tilde) are
$$
\aa= \lambda^{-(N-1)}a,\qquad \tilde{v}=\lambda^{-(N-1)}v, \qquad \tilde{w}=\lambda^{-(N-1)}w, \qquad \BB_i=\lambda^{-(M-1)} B_i,
$$
with $i=1,2$ and $\tilde{C}=C$, respectively. The domain of $\tilde{F}$
contains a ball $\overline B(\tilde\delta)$, with $\tilde\delta=
\esc \bar\delta > \delta_0$.
\item $\tilde{R}(t)=t-\tilde{a}t^{N}+\tilde{b}t^{2N-1}$ with $\tilde{b}=\lambda^{-2N+2} b$.
We further ask  $|\lambda^{-1/\gamma} a|\leq \e$
where $\gamma$ is defined in~\eqref{defalpha}. As a consequence $\aa \leq \e$.
\item Formally we have that
\begin{equation*}
\tilde{F}\circ \tilde{K} - \tilde{K} \circ \tilde{R} = 0.
\end{equation*}
\item Let $\displaystyle \tilde{\CC}=\max_{(x,y,z)\in \overline{B}(\tilde{\delta})}
\norm{\tilde{G}(x,y,z)}$. It is clear that $\tilde{\CC}= \esc \bar{\CC}$ and hence
$\norm{\tilde{G}_k} \leq \esc \bar{\CC} \tilde{\delta}^{-k} $ for all
$k\geq 0$.
\item $\norm{\tilde{K}_l} \leq \mu_0 l!^{\gamma}$
for all $N\leq l \leq l_0$. We recall that $\tilde{K}_l=0$ if $2\leq
l\leq N-1$.
\end{enumerate}
\end{lemma}

We remark that $\sigma := \frac{\tilde{b}}{\tilde{a}^2}$ does not depend on the rescaling parameter $\lambda$.

\begin{lemma}\label{lemmaAl}
The matrices $\tilde{\A}_l$ defined as in~\eqref{defAl} with $\tilde{B_1}$ instead of $B_1$, after the changes of variables
in Lemmas~\ref{lemmaN} and~\ref{scaling}, satisfy
$\tilde{\A}_l = \lambda^{N-1}\A_l$ if $M\geq N$ and $\tilde{\A}_l =
\lambda^{M-1} \A_l$ when $M<N$. Moreover, if $l\geq l_0 \geq 2 |a|^{-1} \Vert B_1\Vert$,
$$
\| \tilde{\A}_l \| \leq 2 \lambda^{N-1} (l|a|)^{-1},\qquad \text{if}\qquad M\geq N.
$$
\end{lemma}

The following technical lemmas are slight variations of lemmas
in~\cite{BH08}. For the reader's convenience, we state and prove
them.
\begin{lemma}\label{technicallemmas} Let $k,\nu\in \N$, $\nu\geq k$ and $\beta \ge \frac{1}{N-1}$.
Let  also
\begin{equation*}
J^1_{k,\nu} = k!^\beta R_{k,\nu}, \qquad \text{where} \qquad R_{k,\nu} = \sum_{\begin{array}{c} \scriptstyle{l_1+\cdots+l_k= \nu} \\
\scriptstyle{l_i \geq 1} \end{array}} \prod_{i=1}^k \RR_{l_i}
\end{equation*}
and $\tilde R_l$ are the coefficients of the polynomial $\tilde R(t) = t -\tilde a t^N + \tilde b t^{2N-1}$.
Let $\sigma= \frac{\tilde b}{\tilde a^2}$ and $m=\frac{\nu-k}{N-1}$.
We have
\begin{equation*}
\begin{cases}
J_{k,k}^1= k!^\beta \\
|J_{k,\nu}^1|  \leq (\nu-N+1)!^{\beta} (\nu-mN+1) |\aa|^m
(1+|\sigma|)^{m/2},& \qquad \text{if}\quad \frac{\nu-k}{N-1}\in \N,\\
J_{k,\nu}^1  = 0, & \qquad \text{ otherwise.}
\end{cases}
\end{equation*}
\end{lemma}
\begin{proof}
Note that $R_{k,\nu}$ is the coefficient of $t^\nu$ of the polynomial $(t-\aa t^{N} +
\bb t^{2N-1})^k$. Then, we can rewrite it as
\begin{equation}\label{Rknu1}
R_{k,\nu}=\sum_{\begin{array}{c} \scriptstyle{m_1+m_2+m_3=k} \\
\scriptstyle{m_1 + N m_2 + (2N-1)m_3=\nu}
\end{array}}\frac{k!}{m_1!m_2 ! m_3!} (-\aa)^{m_2} \bb^{m_3}.
\end{equation}
The conditions on the indices $m_2,m_3$ in the previous formula imply
$(N-1) m_2 + 2(N-1) m_3 = \nu-k$, that is,
$m_2 + 2 m_3 = (\nu-k)/(N-1) \in \N$.
Therefore $R_{k,\nu}=0$ if $m:=(\nu-k)/(N-1) \notin \N$. When $\nu=k$, $m_2=m_3=0$ and $m=0$. Then $R_{k,k}=1$ and
$J_{k,k}^1=k!^{\beta}$. If $m\geq 1$, we reduce~\eqref{Rknu1} to a sum with a single index as
\begin{equation*}
R_{k,\nu} = \sum_{m_3 = 0}^{\left
[\frac{m}{2}\right]}\frac{(\nu-(N-1)m)!}{(\nu-mN +m_3)! (m-2m_3)!
m_3!} (-\aa)^{m-2m_3} \bb^{m_3} .
\end{equation*}
Using
\begin{equation*}
\frac{(\nu-(N-1)m)!}{(\nu-Nm +m_3)!}
\leq \frac{(\nu-(N-1)m)!}{(\nu-Nm)!}
\leq (\nu-(N-1)m)^{m-1} (\nu-Nm+1)
\end{equation*}
and
\begin{equation*}
\sum_{m_3 = 0}^{\left [\frac{m}{2}\right ]}
\frac{|\aa|^{m-2m_3}|\bb|^{m_3}}{(m-2m_3)! m_3!}
 \leq |\aa|^m
\sum_{m_3 = 0}^{\left[\frac{m}{2}\right]} \frac{1}{([m/2]-m_3)!
m_3!} {\left|\frac{\bb}{\aa^2}\right|}^{m_3}   \leq
\frac{1}{\left [\frac{m}{2}\right ]!} |\aa|^m (1+|\sigma|)^{m/2},
\end{equation*}
we get
\begin{equation*}
|R_{k,\nu}|\leq  \frac{1}{\left [\frac{m}{2}\right ]!}
(\nu-(N-1)m)^{m-1} (\nu-Nm+1) |\aa|^m (1 + |\sigma|)^{m/2}.
\end{equation*}

Finally, since $k= \nu-(N-1)m$, using that $m\geq 1$ and that
\begin{align*}
\frac{(\nu-(N-1)m)!^{\beta}
(\nu-(N-1)m))^{m-1}}{(\nu-N+1)!^{\beta}}&=
\frac{(\nu-(N-1)m)^{m-1}}{[(\nu-N+1)\cdots (\nu-(N-1)m+1)]^{\beta}} \\
&\leq \frac{(\nu -(N-1)m)^{m-1}}{(\nu-(N-1)m+1)^{(m-1)(N-1)\beta}}
\leq 1,
\end{align*}
we obtain
\begin{align*}
|J_{k,\nu}^1| & \leq \frac{1}{\left [\frac{m}{2}\right ]!}
(\nu-(N-1)m)!^\beta (\nu-(N-1)m)^{m-1} (\nu-Nm+1) |\aa|^m (1 +
|\sigma|)^{m/2},
\end{align*}
where we use that $(N-1)\beta\geq 1$.
\end{proof}

The next lemma collects two technical results on bounds of some
products of factorials.
\begin{lemma}
\label{lem:J2knu} Let $N\ge 2$, $\beta \ge \frac{1}{N-1}$ and
$N_\beta=N^{\beta(N-1)}$.
\begin{enumerate}
\item[i)] Let $k\geq 1$, $\nu\geq k N$ and
\begin{equation} \label{defMknu}
M_{k,\nu}=
\sum_{\begin{array}{c}
\scriptstyle{l_1+\cdots+l_k=\nu} \\ \scriptstyle{ l_i\geq N}
\end{array}}
(l_1!\cdot \, \cdots \, \cdot l_k!)^{\beta}.
\end{equation}
If $\nu<kN$, the sum in~\eqref{defMknu} is void and we define $M_{k,\nu}=0$.
We have
\begin{equation*}
\begin{cases}
M_{k,\nu}  \leq (\nu-k+1)!^\beta N_{\beta}^{k-1},  &\qquad \nu \geq k N, \\
M_{k,\nu}= 0, &  \qquad \text{ otherwise}.
\end{cases}
\end{equation*}
\item [ii)] Let $k\geq 1$, $\nu \geq k$ and
\begin{equation}\label{defJknu2lema}
J_{k,\nu}^2= \sum_{\begin{array}{c}
\scriptstyle{l_1+\cdots+l_k=\nu} \\ \scriptstyle{l_i=1\text{ or } l_i\geq N}
\end{array}}(l_1!\cdot \cdots \cdot l_k!)^{\beta}, \qquad \nu\geq k.
\end{equation}
If $\nu<k$ the sum in~\eqref{defJknu2lema} is void and we define $J_{k,\nu}^2=0$.
We have
\begin{equation*}
J_{k,\nu}^2\leq \frac{(N_{\beta}+1)^k -1 }{N_{\beta}} (\nu-k+1)!^\beta ,\qquad \nu \geq k.
\end{equation*}
\end{enumerate}
\end{lemma}

\begin{proof}
i) If $kN>\nu$, one has that $M_{k,\nu}=0$. Let us assume that
$kN\leq \nu$. One can check that, if $a,b,c\in \N$ with $b\leq c$,
then $(a+b)! c! \leq b! (a+c)!$. Therefore, for $l_1,l_2,\cdots,
l_k\geq N$ such that $l_1+\cdots +l_k =\nu$ one has that $l_1 ! l_2!
\leq N! (l_1 + l_2 -N)!$, that
\begin{equation*}
l_1! l_2 ! l_3! \leq N! (l_1+l_2-N)! l_3! =N! (l_1+ l_2-2N+N)! l_3!
\leq N!^2 (l_1+l_2+l_3-2N)!
\end{equation*}
and applying this procedure recursively we get
\begin{equation*}
l_1 ! l_2!\cdot \, \cdots \, \cdot l_k! \leq N!^{k-1} (l_1+\cdots +l_k -(k-1)N)! =
N!^{k-1} (\nu - (k-1)N)!.
\end{equation*}
On the other hand it is clear that
\begin{align*}
\#\{l_1+\cdots +l_k=\nu, \; l_i\geq N\} &= \#\{m_1+\cdots+m_k = \nu -kN,\; m_i\geq 0\} \\&=
\comb{\nu-kN+k-1}{k-1}.
\end{align*}
Therefore
$$
M_{k,\nu} \leq N!^{\beta(k-1)} (\nu-(k-1)N)!^{\beta} \comb{\nu-kN+k-1}{k-1}.
$$
Now we use that $N!^{\beta}\leq N^{\beta(N-1)}$ that
$$
 \comb{\nu-kN+k-1}{k-1} = \frac{\nu-kN+k-1}{k-1} \cdot \, \cdots \, \cdot \frac{\nu-kN +1}{1} \leq (\nu-kN+1)^{k-1}
$$
and that
$$
\frac{(\nu -k+1)!}{(\nu -(k-1)N)!}\geq (\nu -(k-1)N+1)^{(k-1)(N-1)}
$$
to obtain
\begin{align*}
M_{k,\nu}
&
\leq N_\beta^{k-1} (\nu-(k-1)N)!^{\beta} (\nu-kN+1)^{k-1} \\
&
\leq N_\beta ^{k-1} (\nu-k+1)!^\beta \frac{(\nu-kN+1)^{k-1}}{(\nu-(k-1)N+1)^{\beta(N-1)(k-1)}}.
\end{align*}
Finally the bound in i) follows because $\beta(N-1) \geq 1$.

ii) For $k= \nu$, $J^2_{k,\nu}= 1$ and the bound is obvious. Assume that $\nu>k$.
Then,
\begin{align*}
J_{k,\nu}^2&=\sum_{i=0}^{k-1}
\comb{k}{i}  M_{k-i,\nu-i}
\leq (\nu-k+1)!^\beta \sum_{i=0}^{k-1} \comb{k}{i}  N_\beta^{k-i-1} \\
&\leq (\nu-k+1)!^\beta \frac{(N_\beta+1)^k-1}{N_\beta}
\end{align*}
and the proof is complete.
\end{proof}

Now we are going to prove that $\tilde{K}(t)=\sum_{l\geq 1} \tilde{K}_l t^l$ is
Gevrey of order $\gamma$. Recall that $\gamma$ was defined in~\eqref{defalpha}.
\begin{proposition}\label{gevprop}
Let  $\bar{\delta}, \bar{\CC}$ be as in Lemma~\ref{scaling}. We take $\delta_0= 2(1+N_\gamma)$, with $N_{\gamma}=N^{\gamma(N-1)}$ and
$\e, \mu_0$ and $l_0$ according to the cases
\begin{itemize}
\item $N\leq M$,
\begin{align*}
\e&=\frac{1}{8 (1+|\sigma|)},\\
\mu_0& =1 \\
l_0&\geq  \max \left \{ 6 \bar{\CC} (1+N)^{M}\bar{\delta}^{-M} (N|a|)^{-1},
 N + \frac{4\mu_0^{N-1}}{N} (1+N)^N \left(1+\frac{2 \bar{\CC}}{|a| \bar{\delta}^N}\right)
 \right \}
\end{align*}
\item $M<N$,
\begin{align*}
\e&= \min\left\{ \frac{1}{8 (1+|\sigma|)}, \frac{1}{4\|B^{-1}_1\|(1+|\sigma|)^{1/2}}\right\},\\
\mu_0 &= \min \left \{ 1, \left [{\displaystyle \frac{N_{\gamma}}{6\bar{\CC}\|B^{-1}_1\|}\left (\frac{\bar \delta}{1+N_\gamma}\right )^{M}}\right ]^{\frac{1}{M-1}} \right \}, \\
l_0&\geq  N + \frac{4\mu_0^{N-1}}{N_\gamma} (1+N_\gamma)^N \left(1+\frac{2 \bar{\CC}}{|a| \bar{\delta}^N}\right ).
\end{align*}
\end{itemize}
Let $\tilde{F}$ rescaled with $\lambda$ depending on $\delta_0, \varepsilon, \mu_0,l_0$ as in Lemma~\ref{scaling} and $\tilde{K}$ be
the formal solution $\tilde{K}(t)=\sum_{j=1}^{\infty} \tilde{K}_j t^j$
of equation $\tilde{F}\circ \tilde{K} - \tilde{K}\circ \tilde{R}=0$. Then
$$
\Vert \tilde{K}_j \Vert \leq \mu_0 j!^{\gamma},\qquad  j\geq 0.
$$
\end{proposition}
\begin{proof}
We use the formulas for $\KK$ and $E$ in
Proposition~\ref{construction} applied to $\tilde{F}$, and hence we
have that $\KK_l=0$ if $2\le l\leq N-1$. By Lemma~\ref{scaling} and
the choice of parameters we have that $\|\KK_j\| \leq \mu_0
j!^{\gamma}$ for $N\leq j\leq l_0$. We will use
Lemmas~\ref{technicallemmas} and \ref{lem:J2knu} with
$\beta=\gamma$.

We assume by induction that $\|\KK_j\| \leq \mu_0 j!^{\gamma}$ for $1\leq j\leq l-1$ for some $l>l_0$.

We start bounding $\tilde K_l^z$. We introduce
$$
H_l^1 =
\sum_{k=2}^{l}
\sum_{\begin{array}{c} \scriptstyle{l_1+\cdots+l_k=l} \\
\scriptstyle{1\leq l_i \leq l-1}
\end{array}} \GG^z_k[\KK_{l_1},\cdots,\KK_{l_k}],\qquad
H_l^2 = \sum_{k=N}^{l-N+1}\KK_k^z \RR_{k,l}
$$
so that $E_l^z = H_l^1- H_l^2$. We have that, by 4 of
Lemma~\ref{scaling}, taking $\lambda$ such that
$\mu_0(N_{\gamma}+1)/\tilde \delta < 1/2$ and using
Lemma~\ref{lem:J2knu},
\begin{align*}
\Vert H_l^1 \Vert  &
\leq \lambda \bar{\CC}\sum_{k=2}^l  \tilde{\delta}^{-k} \mu_0^k   J_{k,l}^2 \leq \frac{\lambda\bar \CC l!^{\gamma} }{N_{\gamma} }\sum_{k=2}^l \left [\frac{\mu_{0} (N_{\gamma}+1)}{\tilde{\delta}}\right ]^k \frac{(l-k+1)!^{\gamma}}{l!^{\gamma}} \\
&\leq 2 \frac{\lambda \bar \CC l!^\gamma}{N_\gamma} \left [\frac{\mu_{0} (N_{\gamma}+1)}{\tilde{\delta}}\right ]^2
\leq \mu_0 l!^{\gamma} \lambda^{-1} \big [ 2 \bar\CC \bar \delta^{-2} (N_\gamma +1)^2 N_\gamma^{-1} \big ].
\end{align*}
Moreover, for $H_l^2$, it is clear that $ \Vert H_l^2 \Vert  \leq
\sum_{k=N}^{l-N+1} \mu_0 k!^{\gamma} | \RR_{k,l}| = \mu_0
\sum_{k=N}^{l-N+1} J_{k,l}^1. $ Then, using
Lemma~\ref{technicallemmas} and writing $m=(l-k)/(N-1)$,
\begin{align*}
\Vert H_l^2 \Vert  &\leq \mu_0 (l-N+1)!^{\gamma}\sum_{m=1}^{\left [ \frac{l-N}{N-1}\right ]} (l-mN+1) |\tilde{a}|^m (1+|\sigma|)^{m/2}\\
&\leq \mu_0 (l-N+1)!^{\gamma} (l-N+1) 2 |\tilde{a}|(1+|\sigma|)^{1/2} \\
&\leq \mu_0 l!^{\gamma} \lambda^{-N+1} \big [ 2|a| (1+|\sigma|)^{1/2} \big ].
\end{align*}
Therefore,  since $\Vert \KK_z^l \Vert \leq  \Vert (C-\Id)^{-1}
\Vert \Vert E_{l}^z \Vert \leq  \Vert (C-\Id)^{-1} \Vert  (\Vert
H_l^1 \Vert + \Vert H_l^2 \Vert) $,
\begin{equation}\label{boundKzl}
\Vert \KK_z^l \Vert \leq \mu_0 l!^{\gamma} \lambda^{-1} \left[2 \bar\CC \bar \delta^{-2} (N_\beta +1)^2 N_\beta^{-1} + \lambda^{-N+2} \big [ 2|a| (1+|\sigma|)^{1/2} \big ]\right ]
\end{equation}
and taking $\lambda$ big enough we obtain $\|\KK_z^l\|\leq \mu_0 l!^{\gamma}$.

To bound $\KK_l^y$ we introduce $H_l^3$ and $H_l^4$ so that
$E_y^{l+L-1} =H_{l}^3 - H_l^4$:
$$
H_l^3 =
\sum_{k=M}^{l+L-1}
\sum_{\begin{array}{c} \scriptstyle{l_1+\cdots+l_k=l+L-1} \\
\scriptstyle{1\leq l_i \leq \min\{l-1,l+L-M\}}
\end{array}} \GG^y_k[\KK_{l_1},\cdots,\KK_{l_k}],
\quad
H_l^4 = \sum_{k=N}^{\min\{l-1, l+L-N\}}\KK_k^y \RR_{k,l+L-1}.
$$

We distinguish two cases, when $L=\min\{N,M\}=N$ and $L=M<N$. First
we deal with the case $L=N$. In this case $\gamma=1/(N-1)$,
$N_\gamma=N$ and $\mu_0=1$. By item~ii) of Lemma~\ref{lem:J2knu}, we
have that
\begin{align*}
\|\tilde{\A}_l H_l^3\| &\leq\lambda^{N}\bar{\CC}\frac{1}{l|a|}\sum_{k=M}^{l+N-1}
J_{k,l+N-1}^2 \frac{1}{\tilde{\delta}^k} \leq  \lambda^{N}\bar{\CC}\frac{1}{N l|a|}
\sum_{k=M}^{l+N-1} (l+N-k)!^{\gamma}
\left (\frac{1+N}{\tilde{\delta}}\right )^{k}\\
&\leq \lambda^{N}\bar{\CC}\frac{1}{N l|a|}(l+N-M)!^{\gamma}
\sum_{k=M}^{l+N-1} \left (\frac{1+N}{\tilde{\delta}}\right )^{k} \\
&\leq \lambda^{N}\bar{\CC}\frac{2}{N l|a|}(l+N-M)!^{\gamma}
\left (\frac{1+N}{\tilde{\delta}}\right )^{M},
\end{align*}
since we are assuming that $\frac{1+N}{\tilde{\delta}} \leq\frac{1+N}{\delta_0}= 1/2$.
Therefore, using that $\tilde{\delta}=\lambda \bar\delta$, that $M\geq N=L$ and that $l\geq l_0$,
$$
\|\tilde{\A}_l H_l^3\| \leq l!^{\gamma} l_0^{-1} \left  [\lambda^{N-M}\bar{\CC}\frac{2}{N |a|}
\left (\frac{1+N}{\bar{\delta}}\right )^{M}\right ] \leq \frac{1}{3} l!^{\gamma}
$$
by definition of $l_0$.

Now we deal with $H_l^4$. It is clear that
$\|H_l^4\| \leq  \sum_{k=N}^{l-1} J_{k,l+N-1}^1$.
Recall that $J_{k,l+N-1}^1\neq 0$ if and only if $k=l+N-1-m(N-1)$
for some $m\in \N$. If $m=1$, then $k=l+N-1-(N-1)=l$ which is a contradiction,
because $k\leq l-1$. Moreover, since $k\geq N$, $m\leq (l-1)/(N-1)$. Therefore, using Lemmas~\ref{lemmaAl} and~\ref{technicallemmas}:
$$
\| \tilde{\A}_l H_l^4\| \leq 2\lambda^{N-1}\frac{1}{|a|l}
\sum_{m=2}^{\left [\frac{l-1}{N-1}\right ]} l!^{\gamma} (l+N-mN) |\aa|^m
(1+|\sigma|)^{m/2}.
$$
Then, since $|\aa (1+|\sigma|)^{1/2}|\leq \e (1+|\sigma|) \leq 1/8$ and $\aa = \lambda^{-(N-1)}a$, we have that
$$
\| \tilde{\A}_l H_l^4\|
\leq 4\lambda^{N-1}\frac{1}{|a|l} l!^{\gamma} (l-N)
|\aa|^2 (1+|\sigma|) \leq l!^{\gamma} \lambda^{-(N-1)}[ 4 |a|(1+|\sigma|)] \leq \frac{1}{3} l!^{\gamma}
$$
if $\lambda$ is big enough. On the one hand, when $N<M$,
$\tilde K_{l}^y =\tilde{\mathcal{A}}_l (H_{l}^3 + H_{l}^4)$ and the previous bounds imply the induction result. On the other hand, when $N=M$,
$\tilde K_{l}^y =\tilde{\mathcal{A}}_l (H_{l}^3 + H_{l}^4 + \tilde{B}_2 \KK_l^z)$. By~\eqref{boundKzl}, the term $\|\tilde{\A}_l \tilde{B}_2 \KK_{l}^z\|$ is bounded by
$l!^{\gamma}/3 $ if $\lambda$ is large enough and therefore, also in this case, we are done.

Now we deal with the case $L=M<N$. Recall that $\gamma=1/(N-M)$. Then
\begin{align*}
\|\tilde{\A}_lH_l^3\| &\leq\lambda^{M}\bar{\CC}\|\tilde{B}_1^{-1}\|\sum_{k=M}^{l+M-1}
J_{k,l+M-1}^2 \frac{\mu_0^k }{\tilde{\delta}^k}
\\&
\leq \lambda^{M}\bar{\CC}\|\tilde{B}_1^{-1}\|
\sum_{k=M}^{l+M-1} \frac{1}{N_\gamma}(l+M-k)!^{\gamma}
\left (\frac{\mu_0(1+N_\gamma)}{\tilde{\delta}}\right )^{k}\\
&\leq \lambda^{M}\bar{\CC} \frac{\|\tilde{B}_1^{-1}\|}{N_\gamma} l!^{\gamma}
\sum_{k=M}^{l+M-1} \left (\frac{\mu_0(1+N_\gamma)}{\tilde{\delta}}\right )^{k}
\\&
\leq \lambda^{M}2 \bar{\CC}\frac{\|\tilde{B}_1^{-1}\|}{N_\gamma} l!^{\gamma}
\left (\frac{\mu_0(1+N_\gamma)}{\tilde{\delta}}\right )^{M}
\end{align*}
since $\mu_0\leq 1$ and $\frac{1+N_\gamma}{\tilde{\delta}} \leq
\frac{1+N_\gamma}{\delta_0}= 1/2$. Moreover, using that $\tilde{\delta}=\lambda \bar \delta$,
$$
\|\tilde{\A}_lH_l^3\| \leq \mu_0 l!^{\gamma} \mu_0^{M-1} \left [2\bar{\CC}\frac{\|\tilde{B}_1^{-1}\|}{N_\gamma}\left (\frac{1+N_\gamma}{\bar{\delta}}\right )^{M}\right ]
\leq \frac{1}{3}  \mu_0 l!^{\gamma}
$$
by definition of $\mu_0$.

Now we deal with $H_l^4$. By induction hypothesis,
$\| H_l^4\| \leq \mu_0  \sum_{k=N}^{l+M-N} J_{k,l+M-1}^1$.
Then, using Lemmas~\ref{lemmaAl} and~\ref{technicallemmas}:
\begin{align*}
\| \tilde{\A}_l H_l^4\| &\leq \mu_0 \lambda^{M-1}\|\tilde{B}_1^{-1}\|
\sum_{m=1}^{\left [\frac{l+M-N-1}{N-1}\right ]} (l+M-N)!^{\gamma} (l+M-mN)
|\aa|^m (1+|\sigma|)^{m/2} \\
&\leq 2\mu_0 \lambda^{M-1}\|\tilde{B}_1^{-1}\| (l+M-N)!^{\gamma} (l+M-N)
|\aa| (1+|\sigma|)^{1/2}\\
&\leq 2 \mu_0 \lambda^{M-N} \|\tilde{B}_1^{-1}\||a|(1+|\sigma|)^{1/2}  (l+M-N)!^{\gamma} (l+M-N)
\end{align*}
since $|\aa| (1+|\sigma|)^{1/2}\leq \e(1+|\sigma|)^{1/2} \leq 1/2$. Now we stress that, since $\gamma=1/(N-M)$,
$$
(l+M-N)!^{\gamma} (l+M-N) \leq l!^{\gamma} \frac{l+M-N}{(l+M-N+1)^{\gamma(N-M)}}\leq l!^{\gamma}.
$$
Therefore, by Lemma~\ref{scaling},
$$
\| \tilde{\A}_l H_l^4\| \leq \mu_0 l!^{\gamma}\lambda^{M-N} \big [ 2
\|\tilde{B}_1^{-1}\|\lambda^{M-N} |a| (1+|\sigma|)^{1/2}
l!^{\gamma}\big] \leq \frac{1}{3} \mu_0 l!^{\gamma},
$$
if $\lambda$ is big enough. Notice that, by Lemma~\ref{lemmaAl}, $\tilde{\A}_l \tilde{B}_2 = \A_l B_2$, so that by~\eqref{boundKzl}
we can take $\lambda$ big enough such that $\|\tilde{\A}_l \tilde{B}_2 \KK_{l}^z\| \leq \mu_0 l!^{\gamma}/3$.
Therefore, we also get in this case that $\|\KK_l^y\| \leq \mu_0 l!^{\gamma}$.

Now we deal with $\KK_l^x$. We decompose $\KK_{l}^x =H_l^5+ H_l^6+H_l^7+H_l^8$ with
\begin{align*}
H_l^5&=\frac{1}{l-N}
\sum_{\begin{array}{c}\scriptstyle{l_1+\cdots+l_N=l+N-1} \\
\scriptstyle{1\leq l_i \leq l-1}
\end{array}} \prod_{i=1}^N \KK^x_{l_i} ,\qquad
H_l^6 =\frac{1}{\aa(l-N)}\sum_{k=N}^{l-1} \KK^x_k R_{k,l+N-1},\\
H_l^7&=\frac{1}{\aa(l-N)}\sum_{k=N}^{l+N-1}
\sum_{\begin{array}{c} \scriptstyle{l_1+\cdots+l_k=l+N-1} \\
\scriptstyle{1\leq l_i \leq l}
\end{array}} G^x_k[\KK_{l_1},\cdots, \KK_{l_k}],
\quad H_l^8=\frac{\big (\tilde{v}^\top \KK_l^y + \tilde{w}^{\top}
\KK_l^z\big )}{\aa (l-N)},
\end{align*}
where in $H_l^7$ we use that $\KK_2^x = \dots = \KK_{N-1}^x = 0$ by
Lemma~\ref{lemmaN} and that $\KK^y_l,\KK^z_l$ are already known. We
notice that, using ii) of Lemma~\ref{lem:J2knu} and that $\mu_0\leq
1$
\begin{align*}
|H_l^5| &\leq \mu_0^N \frac{1}{l-N}\sum_{\begin{array}{c}
\scriptstyle{l_1+\cdots+l_N=l+N-1} \\ \scriptstyle{l_i=1\text{ or}
N\leq l_i \leq l-1}
\end{array}} (l_1!  \cdot \, \cdots \, \cdot l_{N}!)^{\gamma}
\leq \mu_0 l!^{\gamma} \frac{1}{l_0-N} \left [\frac{\mu_0^{N-1} (1+N_\gamma)^N }{N_\gamma} \right ] \\ &\leq \frac{1}{4} \mu_0 l!^{\gamma}
\end{align*}
by definition of $l_0$.
To bound $|H_l^6|$ we use Lemma~\ref{technicallemmas} and we get
\begin{align*}
|H_{l}^6|  \leq& \mu_0 \frac{1}{|\aa|(l-N)} \sum_{k=N}^{l-1} k!^{\gamma}
R_{k,l-N+1} \\
\leq& \mu_0 \frac{1}{|\aa|(l-N)}  l!^\gamma \sum_{m=2}^{\left[\frac{l-1}{N-1}\right]} (l-N(m-1)) |\aa|^m (1+|\sigma|)^{m/2} \\
 \leq&  \mu_0  l!^\gamma \frac{1}{|\aa|} 2 |\aa|^2 (1+|\sigma|)  =  \mu_ 0 l!^\gamma \lambda^{-(N-1)} \big [2 |a| (1+|\sigma|)\big] \leq  \frac{1}{4} \mu_0 l!^{\gamma} ,
\end{align*}
where we used that $|\aa| (1+|\sigma|) \leq \varepsilon (1+|\sigma|) \leq \frac{1}{2}$ and that $\lambda$ is large enough.
To bound $H_l^7$, we recall that $\frac{1+N_\gamma}{\tilde{\delta}} \leq 1/2$,
$\aa= \lambda^{1-N} a$ and that $\tilde{\delta}= \lambda
\bar{\delta}$. Then, by definition of $l_0$:
\begin{align*}
|H_{l}^7| \leq& \frac{\lambda \bar{\CC}}{|\aa|(l-N)}
\sum_{k=N}^{l+N-1} \tilde{\delta}^{-k} \mu_0^k J_{k,l+N-1}^2 \\
\leq& \frac{\lambda \bar{\CC}}{|\aa|(l-N)} \sum_{k=N}^{l+N-1}
(l+N-k)!^{\gamma}
\frac{1}{N_\gamma} \left(\frac{\mu_0(1+N_\gamma)}{\tilde{\delta}}\right)^k \\
\leq& \frac{2\lambda \bar{\CC}}{|\aa|(l-N)N_\gamma} l!^{\gamma}
\left(\frac{\mu_0(1+N_\gamma)}{\tilde \delta}\right)^N
\leq
\mu_0  l!^{\gamma} \frac{1}{l_0-N} \left [\frac{2 \bar{\CC } \mu_0^{N-1} }{|a| N_\gamma} \left(\frac{1+N_\gamma}{\bar{\delta}}\right)^N \right ]\\
\leq &\frac{1}{4} \mu_0  l!^{\gamma}.
\end{align*}

Moreover, since by Lemma~\ref{scaling}, $\tilde{v}=\lambda^{-(N-1)} v$ and $\tilde{w}=\lambda^{-(N-1)}w$, $H_l^8$ can be made smaller than
$\mu_0 l!^{\gamma}/4$ provided $\lambda$ is big enough. Then $|\tilde K_l^x| \leq \mu_0 l!^{\gamma}$.
\end{proof}


\section{A solution of the invariance equation \texorpdfstring{$F\circ K-K\circ R=0$}{Lg}}\label{sec:existence}

In this section we prove Theorem~\ref{mtexistencesolution}, that is,
there exists a real analytic function which is a true solution of
the invariance equation in an appropriate domain and it is
$\gamma$-Gevrey asymptotic to the formal solution $\hat K$ found in
the previous section.

We will use some basic properties about Gevrey functions. A summary
of these properties can be found in~\cite{BH08}. See
also~\cite{Balser94}.

We begin by applying Borel-Ritt's theorem for Gevrey functions to
the formal solution $\hat K$ (\cite[p.17]{Balser94}). Let
$0<\beta<\alpha \pi$ be an opening of a sector. Then there exist
$\rho$ small enough and a $\gamma$-Gevrey real analytic function,
$K_e$, defined on the sector $S(\beta,\rho)$, which is
$\gamma$-Gevrey asymptotic to the formal solution $\hat K$
(see~\eqref{def:sector} for the definition of the sector). Then,
\begin{equation*}
F\circ K_e - K_e \circ R = E
\end{equation*}
being $E$ a real analytic function on $S(\beta,\rho)$
$\gamma$-Gevrey asymptotic  to the identically zero formal series.
As a consequence, for any closed sector
$$
\bar{S}_1:=\bar{S}_1(\bar{\beta},\bar{\rho})= \{t\in \C : 0<|t|\leq \bar{\rho} ,\; |\text{arg} (t)| \leq \bar \beta /2  \} \subset S(\beta,\rho)
$$
there exist $c_0,c$ such that
$$
\|E(t)\| \leq c_0 \text{exp}\left (-c|t|^{-1/\gamma}\right ), \qquad \text{if   }t\in \bar{S}_1.
$$

We look for a real analytic function $H$ defined on $\bar{S}_1$ such
that
\begin{equation}\label{eqHfirst}
F\circ (K_e + H) - (K_e+H) \circ R =0,\qquad \sup_{t\in \bar{S}_1} |H(t)|  \text{exp}\left (c|t|^{-1/\gamma}\right ) <+\infty.
\end{equation}
For that we rewrite~\eqref{eqHfirst} as a fixed point equation. Let
us to introduce $\hat C(t)$ and $\mathcal{N}$ as:
$$
\hat C(t)= \left ( \begin{array}{ccc} \Id & 0 &0 \\ 0 & \Id + B_1 [K_e^x(t)]^{M-1} & 0 \\ 0 & 0 & C \end{array}\right ) ,
\quad \mathcal{N}(H) =F(K_e + H) - F(K_e)- \hat C(t) H.
$$
Then the equation~\eqref{eqHfirst} becomes
\begin{equation}\label{eqHsecond}
\hat C(t) H(t)  -H\circ R(t)= -E(t) -\mathcal{N}(H)(t).
\end{equation}

We introduce $S_1(\bar \beta,\bar \rho)= \text{int}(\bar{S}_1(\bar \beta,\bar \rho))$ and the Banach spaces
$$
\mathcal{X}_{\ell,m}  = \{H:\bar{S}_1(\bar \beta,\bar \rho)\cup\{0\} \to \C^{m},\;  \mathcal{C}^0, \text{ real analytic in }
S_1 (\bar \beta,\bar \rho)  \text{ and }  \|H\|_{\ell} <+\infty\}
$$
with
$$
\|H\|_{\ell}:=\sup_{t\in \bar{S}_1} \|H(t)\|  |t|^{-\ell} \text{exp}\left (c|t|^{-1/\gamma}\right ).
$$

It is straightforward to check that,  if $H_1,H_2$ are
$\mathcal{C}^0$ functions in $\bar S(\bar \rho, \bar \beta) \cup
\{0\}$, satisfying that $H_1(0)=H_2(0)=0$, then, denoting $\Delta
H(t) = H_1-H_2$,
\begin{align}\label{boundN}
\|\mathcal{N}^{x,y}(H_1)(t) \| \leq& a_1 \|H_1(t)\|^2 + |t|^{M-1} \big (a_2 |t| + a_3 |t|^{N-M} \big)\|H_1(t)\| \notag \\  &+ a_4 \|B_2 \| \|H_1(t)\|, \notag\\
\|\mathcal{N}^z(H_1)(t) \| \leq & b_1 \|H_1(t)\|^2 + b_2 |t| \|H_1(t)\|, \notag\\
\|\mathcal{N}^{x,y} (H_1)(t)-\mathcal{N}^{x,y}(H_2)(t)\| \leq & a_1 \| \Delta H(t)\|^2 \\ &+
|t|^{M-1} \big (a_2 |t| + a_3 |t|^{N-M}  + a_4 \|B_2 \| \big )\| \Delta H(t)\|, \notag\\
\|\mathcal{N}^{z} (H_1)(t)-\mathcal{N}^{z}(H_2)(t)\| \leq & b_1 \|\Delta H(t)\|^2+ b_2 |t| \|\Delta H(t)\|. \notag
\end{align}
To prove the above inequalities we take into account that $K_e^y,K_e^z =\O(|t|^2)$ as well as the form~\eqref{system} of $F$.

We observe that, by scaling the variable $z$, the norm $\|B_2\|$ is as small as we need. In addition, the matrix $C$ does not change
with this scaling.

We are forced to distinguish two cases according to the different values of $M$ and $N$.
\subsection{The case \texorpdfstring{$M\geq N$}{Lg}}
Recall that we are assuming that $\Spec C\subset \{ z\in \C :
|z|\geq 1\}$. In this case we reinterpret~\eqref{eqHsecond} as the
fixed point equation
\begin{equation}\label{fpeMgeN}
H(t) = \mathcal{F}(H)(t):=(\hat C(t) )^{-1}\left [H\circ R(t) -E(t) -\mathcal{N}(H)(t)\right ]
\end{equation}
which is, essentially, the same as the one considered in~\cite{BH08}.
A crux point is that, if $H\in \mathcal{X}_{0,1+d+d'}$, then
\begin{equation}\label{HcompR}
\|H\circ R \|_{0} \leq e^{-\frac{a}{2} c (N-1) \cos \lambda } \|H \|_0
\end{equation}
so that this term is contracting. Following the steps
in the mentioned work, one can easily check that, taking
$0<\beta < \alpha \pi$ and $\rho$ small enough, the fixed point equation~\eqref{fpeMgeN}
has a unique solution belonging to the Banach space $\mathcal{X}_{0,1+d+d'}$
for any $\bar \rho,\bar \beta$ such that
$\bar{S}_1(\bar \beta,\bar \rho) \subset S(\beta,\rho)$.
As a consequence, the invariance condition~\eqref{eqHfirst} can be solved and the solution $K_e+H$
is analytic in the sector $S(\beta,\rho)$ and $\alpha$-Gevrey asymptotic to the formal solution
$\hat K$.

\subsection{The case \texorpdfstring{$M<N$}{Lg}}
When $M<N$ the strategy developed for the case $M\geq N$ can not be
applied. In this case bound~\eqref{HcompR} is not longer true.
Indeed, as shown in Lemma~\ref{controlRj} below (see
also~\cite{BH08}), $|R(t)| \leq |t|(1+  \nu |t|^{N-1})^{-\alpha}$
with $\nu>0$. Then, if $H\in \mathcal{X}_{0,1+d+d'}$,
$$
e^{c |t|^{-1/\gamma}}\|H\circ R(t)\| \leq \|H\|_{0} e^{-c|R(t)|^{-(N-M)}-c|t|^{-(N-M)}} \leq \|H\|_{0}
e^{-c \frac{a}{2} (N-1)\cos \lambda |t|^{M-1})}.
$$
This implies that the term $H\circ R$ is not contracting. For this
reason we rewrite \eqref{eqHsecond} as another fixed point equation.
We recall that, when $M<N$, $\gamma=1/(N-M)$ and we are assuming
that $\Spec B_1 \subset \{z \in \C: \Re z>0\}$ and that $\Spec C
\subset \{z \in \C: |z|>1\}$.

First we define an appropriate norm in $\C^{1+d+d'}$. We take a norm
in $\C^{d'}$ such that $\| C^{-1} \|_{d'} <1$. Notice that, since
$\Spec B_1 \subset \{z\in \C : \Re z>0\}$, there exists a norm in
$\C^{d}$ such that $\| \Id - B_1 t^{M-1} \|_d <1 -\mu |t|^{M-1} \leq
1$. This follows from the fact that $\Id - B_1 t^{M-1}$ is in Jordan
form if $B_1$ is in Jordan form as well. Therefore, since
$K_e^{x}(t) = t+\O(t^2)$, taking $\rho$ small enough,
$$
\| \big (\Id - B_1 [K_e^x(t)]^{M-1}\big )^{-1} \|_d \leq 1.
$$
We finally define
$$
\|(x,y,z)\|=\max\{|x|, \|y\|_d, \| z \|_{d'}\}.
$$
If necessary, we will write $\|(x,y)\| = \max\{|x|,\|y\|_d\}$.

We rewrite equation~\eqref{eqHsecond} as:
\begin{equation}\label{eqHthird}
\mathcal{G}(H)(t)=-E(t) + \mathcal{N}(H)(t) +(0,0,H^z \circ R(t))^\top,
\end{equation}
with $\mathcal{G}(H)(t)=\hat C(t) H(t) -(H^x\circ R(t), H^y \circ
R(t),0)$. The usual way to proceed is: i) to find a formal inverse,
$\mathcal{S}$, of the linear operator $\mathcal{G}$, ii) to prove
that $\mathcal{S}$ is continuous in appropriate Banach spaces and
iii) to write equation~\eqref{eqHthird} as a fixed point equation
and to apply the fixed point theorem.

The formal operator
$\mathcal{S}=(\mathcal{S}^x, \mathcal{S}^y, \mathcal{S}^z)$, acting on analytic functions $T$, defined by
\begin{align*}
\mathcal{S}^x (T)&= \sum_{j\geq 0}  T^x  \circ R^j, \\
\mathcal{S}^y (T)&= \sum_{j\geq 0}  \prod_{i=0}^j (\Id + B_1 [K_e \circ R^i]^{M-1})^{-1} T^y \circ R^j, \\
\mathcal{S}^z (T)&= C^{-1} T^z
\end{align*}
is the formal inverse of $\mathcal{G}$. The proof of this fact is straightforward.
We (formally) rewrite equation~\eqref{eqHthird} as:
\begin{equation}\label{fpeM<N}
H = \mathcal{F}(H) := -\mathcal{S}\big (E+\mathcal{N}(H) + (0,0,H^z \circ R)^\top\big )
\end{equation}

To obtain accurate bounds for $\mathcal{S}$, we need precise estimates on the convergence of the iterates $R^k(t)$ for
$t\in S(\beta,\rho)$. Given $\nu>0$, let $\mathcal{R}_{\nu}:[0,\infty) \to \R$ be defined by
$$
\mathcal{R}_{\nu}(u)=\frac{u}{(1+\nu u^{N-1})^{\alpha}}.
$$
\begin{lemma}\label{controlRj}
Let $R: S(\beta,\rho)\to \C$ be a map of the form $R(t)=t-at^N + \mathcal{O}(|t|^{N+1})$, $a>0$. Assume that
$\beta < \alpha \pi$. Then, for any $0<\nu<a(N-1) \cos \lambda$, with $\lambda=(N-1)\beta/2$, there exists $\rho$
small enough such that
$$
|R^{k}(t)|\leq \mathcal{R}_{\nu}^k (|t|) = \frac{|t|}{\big (1+ k \nu |t|^{N-1}\big )^{\alpha}}.
$$
In addition, $R$ maps $S(\beta,\rho)$ into itself.
\end{lemma}
\begin{proof}
Writing $t=|t| e^{i\theta}$, the computation of the modulus of $R(t)$ gives
$$
|R(t)| =|t|\big [1 -  a |t|^{N-1} \cos (N-1) \theta + \mathcal{O}(|t|^{N})\big ]
$$
and $\mathcal{R}_{\nu}(u)=u(1- \alpha \nu u^{N-1} + \mathcal{O}(u^{2N-2}))$. Therefore, since $a \cos (N-1) \beta/2> \alpha \nu$ and
$\rho$ is small enough, $|R(t)| \leq \mathcal{R}_{\nu}(|t|)$, for all $t\in S(\beta,\rho)$.

Since $\mathcal{R}_{\nu}$ is the flow time $1$ of the one dimensional equation $\dot{u}=-\alpha \nu u^N$, i.e. $\mathcal{R}_{\nu}(u)=\varphi(1,u)$, then
$\mathcal{R}_{\nu}^{k}$ is the flow time $k$ of the same equation, that is:
$$
\mathcal{R}_{\nu}^{k} (u)= \varphi(k,u) = \frac{u}{\big (1+k\nu u^{N-1}\big )^{\alpha}}.
$$
Using that $\frac{d}{du} \mathcal{R}_{\nu}(u)>0$, it is easy to prove by induction that $|R^k(t)|\leq \mathcal{R}_{\nu}^k(|t|)$.

To prove that $R(S(\beta,\rho))\subset S(\beta,\rho)$ is straightforward, see~\cite{BH08}.
\end{proof}

Now we deal with the linear operator $\mathcal{S}$.
\begin{lemma}\label{linopS} Let $0<\beta<\alpha \pi/2$, $0<\nu<a(N-1) \cos \lambda$ and $\ell,\ell'\in \mathbb{R}$.
If $\rho$ is small enough, then, for any $\bar \beta \in (0,\beta)$
and $\bar \rho\in (0,\rho)$, $\mathcal{S}$ is a well defined, linear
and bounded operator from $\mathcal{X}_{\ell,1+d}\times
\mathcal{X}_{\ell',d'}$ to $\mathcal{X}_{\ell-M+1,1+d} \times
\mathcal{X}_{\ell',d'}$. In addition,
$$
\| \mathcal{S}^{x,y}(T^{x,y})\|_{\ell-M+1} \leq   \frac{2}{c \alpha(N-M)\nu} \| T^{x,y}\|_{\ell},\qquad \|\mathcal{S}^{z} (T^z)\|_{\ell'} \leq
\|C^{-1}\| \| T^z \|_{\ell'}.
$$
\end{lemma}
\begin{proof}
Let $T$ be a function belonging to $\mathcal{X}_{\ell,1+d+d'}$.
Since $\mathcal{S}^z (T^z) = C^{-1} T^z$, the claim is clear.
We have that
\begin{align*}
\|\mathcal{S}^{x,y}&(T^{x,y})(t) \|\leq \sum_{j\geq 0} \|T^{x,y}(R^j(t))\| \leq  \sum_{j\geq 0} \|T^{x,y}\|_{\ell} |R^j(t)|^{\ell} \text{exp}\left (-c |R^j(t)|^{-1/\gamma}\right ) \\
&\leq \sum_{j\geq 0} \frac{ \|T^{x,y}\|_{\ell} |t|^{\ell}}{\big (1+j \nu |t|^{N-1}\big)^{\alpha \ell}}\text{exp}
\left (-\frac{c}{|t|^{1/\gamma}} (1+j\nu |t|^{N-1}\big)^{\alpha/\gamma}\right )\\
&\leq \|T^{x,y}\|_{\ell} |t|^{\ell} \left ( e^{-c|t|^{-(N-M)}} + \int_{0}^{\infty}
\frac{e^{ -\frac{c}{|t|^{(N-M)}} \left(1+x\nu |t|^{N-1}\right)^{\alpha(N-M)}}}{{\big (1+x\nu |t|^{N-1}\big)^{\alpha \ell}}}\, dx\right ).
\end{align*}
Let $I(t)$ be the integral in the right hand side of the last
inequality. By performing the change of variables
\begin{align*}
v & =\big (1+x\nu |t|^{N-1}\big)^{\alpha (N-M)} ,
\\ dv &=  \alpha(N-M)\nu |t|^{N-1}
\big (1+x\nu |t|^{N-1}\big)^{\alpha (N-M)-1} dx\\ &= \alpha(N-M)\nu|t|^{N-1} v^{-(M-1)/(N-M)}dx
\end{align*}
we have that
\begin{align*}
I(t) &= \frac{1}{\alpha (N-M) \nu |t|^{N-1} }\int_{1}^{\infty} e^{-\frac{c}{|t|^{N-M}} v} v^{-(\ell-M+1)/(N-M)}\, dv \\
&= \frac{|t|^{-(\ell-M+1)/(N-M)}}{\alpha(N-M) \nu |t|^{M-1}} \int_{|t|^{-(N-M)}}^{\infty} e^{-c w} w^{-(\ell-M+1)/(N-M)}\, dw.
\end{align*}
Integrating by parts we easily obtain
$$
I(t) \leq \frac{1}{c \alpha(N-M) \nu |t|^{M-1}} e^{-\frac{c}{|t|^{N-M}}} + \frac{|\ell -M+1|}{N-M} |t| I(t)
$$
and the claim is proven since we can take $|t|<\rho$ small enough.
\end{proof}

It is clear that $E\in \mathcal{X}_{0,1+d+d'}$ for any $0<\bar \rho<\rho$ and $0<\bar \beta < \beta < \alpha \pi$.
By Lemma~\ref{linopS}, $\mathcal{S}(E) \in \mathcal{X}_{-M+1,1+d} \times \mathcal{X}_{0,d'}
\subset \mathcal{X}_{-M+1,1+d+d'}$.
We introduce
$$
\varrho = 2  \Vert \mathcal{S}(E)\Vert_{-M+1},\qquad D= \frac{2}{c \alpha(N-M) \nu}.
$$

\begin{lemma} Let $0<\beta<\alpha \pi$, $0<\nu < a(N-1) \cos \lambda$ and $\rho$ small enough such that
the conclusions of Lemmas~\ref{controlRj} and~\ref{linopS} hold true.
For any $0<\bar \beta <\beta$ and $0<\bar \rho <\rho$, we introduce
$\mathcal{B}_{-M+1}(\varrho)\subset \mathcal{X}_{-M+1,1+d+d'}$ the closed ball of radius $\varrho$.

Then, if $\rho$ is small enough, the fixed point equation~\eqref{fpeM<N}, $H=\mathcal{F}(H)$, has a unique solution $H\in \mathcal{B}_{-M+1} (\varrho)$.
\end{lemma}
\begin{proof}  Let $H\in \mathcal{B}_{-M+1}(\varrho)$. We notice again that, by means of a scaling in the $z$-variable, $\|B_2\|$ is small
enough. In addition, since $N>M$, $|t|\geq |t|^{N-M}$.
By using bound~\eqref{boundN} of $\| \mathcal{N}(H)\|$, we have that, for $|t|\leq \bar \rho<\rho$,
\begin{align*}
\| \mathcal{N}^{x,y}(H) \|_0 &\leq  \left (\varrho a_1 |t|^{-(M-1)} e^{-c|t|^{-(N-M)}} + (a_2+a_3) |t|+a_4 \|B_2 \| \right )\| H \|_{-M+1}
\\ & \leq (\rho^{1/2} + a_4 \|B_2 \|) \| H \|_{-M+1}, \\
\| \mathcal{N}^{z}(H)\|_{-M+1} & \leq \left (\varrho b_1 |t|^{-(M-1)} e^{-c|t|^{-(N-M)}} +b_2|t|\right )\| H \|_{-M+1} \leq \rho^{1/2} \| H \|_{-M+1}
\end{align*}
if $\rho$ is small enough.
Hence, by Lemma~\ref{linopS}
$$
\|\mathcal{S}(\mathcal{N}(H))\|_{-M+1} \leq \max \left \{ D , \|C^{-1} \| \right \}
(\rho^{1/2} + a_4 \|B_2 \|) \| H \|_{-M+1}.
$$
We take $\rho$ and $\|B_2\|$ small enough, such that
\begin{equation}\label{rhoB2small}
\max \left \{ D , \|C^{-1} \| \right \}
(\rho^{1/2} + a_4 \|B_2 \|) \leq \frac{1}{2} (1-\|C^{-1}\|).
\end{equation}
Moreover, since $\|\mathcal{S}^z (H^z \circ R)\|_{M+1} \leq \|C^{-1} \| \|H^z \|_{-M+1}$,
\begin{align*}
\|\mathcal{F}(H)\| & \leq \|\mathcal{S}(\mathcal{N}(H))\|_{-M+1} + \|\mathcal{S}^z (H^z \circ R)\|_{-M+1} \\
&\leq \frac{1}{2}(1+ \|C^{-1}\|)  \| H \|_{-M+1} <  \| H \|_{-M+1}.
\end{align*}
We have proven that $\mathcal{F}(\mathcal{B}_{-M+1}(\varrho) )\subset \mathcal{B}_{-M+1}(\varrho)$.

Now we check that $\mathcal{F}$ is contractive. Indeed, let  $H_1,H_2\in \mathcal{B}_{-M+1}(\varrho)$ be two functions in $\mathcal{B}_{-M+1}(\varrho)$.
Again using~\eqref{boundN},
$$
\|\mathcal{S}(\mathcal{N}(H_1)) - \mathcal{S}(\mathcal{N}(H_2))\|_{-M+1} \leq \max \left \{ D , \|C^{-1} \| \right \}
(\rho^{1/2} + a_4 \|B_2 \|) \| H_1-H_2\|_{-M+1}
$$
and, since $\|\mathcal{S}^z (H^z_1 \circ R-H^z_2\circ R)\|_{M+1} \leq \|C^{-1} \| \|H^z_1-H_2 \|_{-M+1}$, we obtain
$$
\|\mathcal{F}(H_1)-\mathcal{F}(H_2)\|_{-M+1} \leq \frac{1}{2} (1+\|C^{-1}\|)\| H_1 -H_2\|_{-M+1},
$$
using~\eqref{rhoB2small}.
\end{proof}

We deduce from this lemma that equation~\eqref{eqHfirst} is satisfied for $H \in \mathcal{X}_{-M+1,1+d+d'}$. Therefore, the function $K=K_e+H$ is a solution
of $F\circ K = K\circ R$, analytic in $S(\beta,\rho)$ and with $\hat K$ as its asymptotic $\gamma$-Gevrey series.
This proves Theorem~\ref{mtexistencesolution} for the case $M<N$.

\appendix

\section{Proof of Proposition~\ref{prop:FormaF}}\label{A1}

We write $\varphi = \varphi_{\leq  r} + \varphi_{>r}$ being $\varphi_{\leq r}$ the Taylor decomposition of $\varphi$ up to order $r$.
Note that $\varphi(0)=0$.
We also will use $\mathcal{N}_{\leq r}$ and the notation introduce in Section~\ref{sec:notation}.

We perform the normal form procedure (using the struture of eigenvalues) to assure that $\mathcal{N}^z(x,y,0) =\O(\|(x,y)\|^N)$, the change of variables
$$
(\bar{x},\bar{y},\bar{z})=(x, y - \varphi^y_{\leq r}(x), z- \varphi^z_{\leq  r}(x))
$$
and the blow up
$$
\bar{x} = u ,\qquad \bar{y}= u^m v,\qquad \bar{z} =u^n w
$$
for some $n,m\in \mathbb{N}$ to be determined later.
We obtain the new map $F=(F^u, F^v, F^w)$:
\begin{equation*}
\begin{aligned}
{F}^u (u,v,w) =& u+\mathcal{N}^x(u, u^{m} v+\varphi_{\leq r}^y (u), u^n w+\varphi_{\leq r}^z (u)), \\
{F}^v (u,v,w)=&\frac{1}{\big ({F}^u (u,v,w)\big )^m} \left [u^m v +\mathcal{N}^y(u, u^{m} v+\varphi_{\leq r}^y (u),
u^n w+\varphi_{\leq r}^z (u))\right. \\  &\left . +\varphi_{\leq r}^y (u) - \varphi_{\leq  r}^y (F^{u}(u,v,w))\right ],\\
{F}^w (u,v,w)=&\frac{1}{\big ({F}^u (u,v,w)\big )^n} \left [u^n w+\mathcal{N}^z(u, u^{m} v+\varphi_{\leq r}^y (u),
u^n w+\varphi_{\leq r}^z (u))\right. \\  &\left . +\varphi_{\leq r}^z (u) -\varphi_{\leq  r}^z (F^{u}(u,v,w))\right ].
\end{aligned}
\end{equation*}
We have that $F$ is a real analytic map and that
$$
{F}^u (u,v,w) = u +\mathcal{N}^x_{\leq r}(u,u^m v +\varphi_{\leq r}^y(u),u^n w+\varphi_{\leq r}^z(u)) + \o(|u|^r),
$$
where $\mathcal{N}_x^{\leq r}$ is a polynomial of degree $r$. Then,
$$
F^u(u,v,w) = u + \mathcal{N}^x_{\leq r}(u,\varphi_{\leq r}^y(u), \varphi_{\leq r}^z(u)) + \O(|u|^{m+1}\|v\|, |u|^{n+1}\|w\|) +
\o(|u|^r).
$$
In addition, since
\begin{equation*}
\mathcal{N}^x_{\leq r}(u,\varphi_{\leq r}^y(u), \varphi_{\leq r}^z(u))-\mathcal{N}(u,\varphi^y(u),\varphi^z(u))=
\o(|u|^{r+1}),
\end{equation*}
taking $m+2\geq N$ and $n+2\geq N$ and using~\eqref{hipprop}
$$
F^{u}(u,v,w) = u-f^u(u,v,w) u:= u- \big(au^{N-1} + \O(|u|^N) + \O(|u|^{m}\|v\|, |u|^{n}\|w\|)\big )u.
$$
Therefore, $F^u$ satisfies the form of the first component
in~\eqref{system}.

Now we deal with $F^v$. We first note that, since $\varphi$ is invariant, we have that
$$
\varphi_{\leq r}(u) + \mathcal{N}^y(u,\varphi^y_{\leq r}(u),\varphi^z_{\leq r}(u))-\varphi_{\leq r}^y(F^u(u,0,0))
=\o(|u|^r).
$$
Secondly, we observe that, using the mean's value theorem
$$
F^v(u,v,w) = \frac{v}{(1+f^u(u,v,w))^m} + \bar{B}_1(u,v,w) v + u^{n-m} \bar{B}_2(u,v,w) w + \o(\|u|^{r-m}),
$$
$\bar{B}_1, \bar{B}_2$ being matrices with every entry of order at least $\O(\|(u,u^m v,u^n w)\|)$.
Note that,
\begin{align*}
\frac{v}{(1+f^u(u,v,w))^m} =& v + m a u^{N-1} v + \O(|u|^N \|v\|)+ \mathcal{O}(|u|^m\|v\|^2)+ \mathcal{O}(|u|^n\|w\|^2),\\
\bar{B}_1(u,u^m v,u^n w) v =& \bar{B}_1(u,0,0) v + \mathcal{O}(|u|^m\|v\|^2)+ \mathcal{O}(|u|^n\|w\|^2), \\
\bar{B}_2(u,u^m v,u^n w) w =& \bar{B}_2(u,0,0) w + \mathcal{O}(|u|^m\|v\|^2)+ \mathcal{O}(|u|^n\|w\|^2).
\end{align*}
Let $M_1,M_2 \in \mathbb{N}$, $M_1,M_2\geq 2$ and $\hat B_1, \hat B_2$ real matrices be such that $\bar{B}_i(u,0,0)=\hat B_i u^{M_i-1} + \O(|u|^{M_i})$, $i=1,2$.
Eventually, $\hat B_1, \hat B_2$ can be the zero matrix if either $\bar B_1(u,0,0)$ or $\bar B_2(u,0,0)$ are flat at $u=0$. In this case one can take
either $M_1\geq N $ or $M_2\geq N$.

We have then that
\begin{align*}
F^v(u,v,w) = &v+ \big (ma u^{N-1}\Id + u^{M_1-1} \hat B_1)v + u^{M_2-1} \hat B_2 w +\O(|u|^{M_1} \|v\|) \\ &+\O(|u|^{M_2} \|w\|) +
\O(|u|^N \|v\|) + \mathcal{O}(|u|^m\|v\|^2)+ \mathcal{O}(|u|^n\|w\|^2) + \o(\|u|^{r-m})
\end{align*}
The result follows with $M=\min\{M_1,N\}$ taking $m+2\geq \max\{M_1,N\}$, $n=m+\max\{0,M-M_2\}\geq m$ and $B_1$ and $B_2$ adequately.

Finally we deal with $F^w$. We recall that $\mathcal{N}^z(x,y,0) =\O(\|(x,y)\|^N)$ and we notice that $\varphi^{z}(x) = \O(|x|^N)$. Indeed, $\varphi^z$ satisfies
$$
C\varphi^{z}(x) + \mathcal{N}^z(x,\varphi^y(x),\varphi^z(x)) =\varphi^z(x + \mathcal{N}^x(x,\varphi^y(x),\varphi^z(x))
$$
or, taking into account condition~\eqref{hipprop},
\begin{align*}
\left (C - \Id + \int_{0}^1 \partial_z \mathcal{N}^z(x,\varphi^y(x),\lambda \varphi^z(x)) \, d \lambda \right )\varphi^z(x)  =
& - \mathcal{N}^z (x,\varphi^y(x),0)-\varphi^{z}(x)  \\
&+\varphi^z(x + \mathcal{N}^x(x,\varphi^y(x),\varphi^z(x))) \\
=& \mathcal{O}(|x|^N) .
\end{align*}
Since the matrix in the left hand side is invertible if $|x|$ is small enough, $\varphi^z(x)=\O(|x|)^N$.

Performing analogous computations as the ones for $F^v$ and taking into account that $\varphi^z_{\leq r} (x)=\O(|x|^N)$ and that $\mathcal{N}^z(x,y,0) =\O(\|(x,y)\|^N)$ one obtains that
$$
F^w(u,v,w) = Cw+\hat{C}_1 u^{N+m-n-1} v + \O(\|u,v,w\|^2)= Cw +  \O(\|u,v,w\|^2)
$$
and the proof is complete since $N+m-n-1\geq 1$.

\begin{remark}
As a consequence of the proof, Proposition~\ref{prop:FormaF} holds
true if $\mathcal{F} \in \mathcal{C}^{r}$ with $r$ big enough
(including the $\mathcal{C}^{\infty}$ case). In addition the map of
the form~\eqref{system} is also $\mathcal{C}^r$.
\end{remark}

\section{Proof of Proposition~\ref{prop:trickylowerbounds}}
\label{proofofpropositionprop:trickylowerbounds}

Since $\text{graph}\, \hat \varphi$ is the formal invariant manifold
of $F$, it satisfies
\begin{equation}
\label{eq:inveqgraph} \hat \varphi(F^x(x,\hat \varphi(x))) - F^{y,z}
(x,\hat \varphi(x)) = 0.
\end{equation}
Let $\tilde g(x) = \varphi_{\le p} (F^x(x,\varphi_{\le p} (x))) -
F^{y,z} (x,\varphi_{\le p} (x))$. We claim that $\tilde g(x) =
\O(\|x\|^{p+1})$. Indeed, using that $\varphi_1 = 0$ (since
$\text{graph}\, \hat \varphi$ is tangent to the $x$-direction), we
have that
\[
\begin{aligned}
\hat  \varphi(F^x(x,\hat \varphi(x)))  = & \varphi_{\le p}
(F^x(x,\varphi_{\le p}
 (x))) + [\varphi_{\le p}(F^x(x,\hat \varphi(x)))-\varphi_{\le p} (F^x(x,\varphi_{\le p}
 (x)))] \\
  & + \varphi_{>p} (F^x(x,\hat \varphi (x))) \\
  = & \varphi_{\le p} (F^x(x,\varphi_{\le p}(x))) + \O(\|x\|^{N+p})
  + \O(\|x\|^{p+1})
 \end{aligned}
\]
and
\[
\begin{aligned}
F^{y,z} (x,\hat \varphi(x)) = & F^{y,z} (x,\varphi_{\le p} (x)) +
[F^{y,z}
(x,\hat \varphi(x))- F^{y,z} (x,\varphi_{\le p} (x))] \\
 = & F^{y,z} (x,\varphi_{\le p} (x)) + \O(\|x\|^{p+1}).
\end{aligned}
\]
Hence, taking into account~\eqref{eq:inveqgraph}, the claim follows.

We define $G$ by $G^x = F^x$ and $G^{y,z} = F^{y,z} + \tilde g$.
Clearly, $F(x,y,z) - G(x,y,z) = \O(\|(x,y,z)\|^{p+1})$. We only need
to check that $\text{graph}\, \varphi_{\le p}$ is invariant by~$G$,
which is true since, by the definition of $\tilde g$ and $G$,
\begin{align*}
\varphi_{\le p} (G^x(x,\varphi_{\le p}(x))) &= \varphi_{\le p}
(F^x(x,\varphi_{\le p}(x)))
= F^{y,z} (x,\varphi_{\le p} (x)) + \tilde g(x) \\ &= G^{y,z}
(x,\varphi_{\le p} (x)),
\end{align*}
which concludes the proof.

\subsection*{Acknowledgments}
I.B and P.M. have been partially supported by the Spanish Government
MINECO-FEDER grant MTM2015-65715-P and the Catalan Government grant
2014SGR504. The work of E.F. has been partially supported by the
Spanish Government grant MTM2013-41168P and the Catalan Government
grant 2014SGR-1145.

\bibliography{references}
\bibliographystyle{alpha}

\end{document}